\documentclass[11pt, oneside]{amsart}   

\pdfoutput=1

\usepackage[centering]{geometry}                		
\usepackage{amssymb}
\usepackage{xparse}

\newcommand\PSL{\mathrm{PSL}}
\newcommand\SL{\mathrm{SL}}
\newcommand\SO{\mathrm{SO}}

\newcommand\PSU{\mathrm{PSU}}

\newcommand\dd{\mathrm{d}}
\newcommand\pl{\mathrm{pl}}

\newcommand{\p}[2]{\frac{\partial #1}{\partial #2}}
\newcommand\vol{\mathop{\mathrm{vol}}}
\newcommand\Ei{\mathop{\mathrm{Ei}}}
\newcommand\hol{\mathop{\mathrm{hol}}}
\newcommand\supp{\mathop{\mathrm{supp}}}
\newcommand\sump{\sideset{}{^P}\sum}
\newcommand\sgn{\text{sgn}}

\theoremstyle{plain}
\newtheorem{theorem}{Theorem}[section]
\newtheorem{corr}[theorem]{Corollary}
\newtheorem{lemma}[theorem]{Lemma}
\newtheorem{prop}[theorem]{Proposition}

\theoremstyle{definition}
\newtheorem{remark}[theorem]{Remark}

\usepackage{xcolor}
\definecolor{couleur_cite}{rgb}{0.05,.4,0.05}
\definecolor{couleur_link}{rgb}{0.05,0.05,0.4}
\definecolor{couleur_url}{rgb}{0.5,0,0}
\usepackage[hyperfootnotes=false]{hyperref}
\hypersetup{hypertexnames=false,colorlinks=true,citecolor=couleur_cite,linkcolor=couleur_link,urlcolor=couleur_url,
pdfstartview=FitH, pdfauthor=Lindsay Dever and Djordje Milicevic, pdftitle=Ambient prime geodesic theorems on hyperbolic 3-manifolds}

\title[Ambient prime geodesic theorems on hyperbolic 3-manifolds]{Ambient prime geodesic theorems\\ on hyperbolic 3-manifolds}
\author{Lindsay Dever}
\address{Bryn Mawr College, Department of Mathematics, 101 North Merion Avenue, Bryn Mawr, PA 19010, USA}
\email{lmdever@brynmawr.edu}
\author{Djordje Mili\'cevi\'c}
\address{Max-Planck-Institut f\"ur Mathematik, Vivatsgasse 7, D-53111 Bonn, Germany}
\email{dmilicev@mpim-bonn.mpg.de}
\address{Bryn Mawr College, Department of Mathematics, 101 North Merion Avenue, Bryn Mawr, PA 19010, USA}
\email{dmilicevic@brynmawr.edu}
\thanks{D.M. was supported by National Science Foundation Grant DMS-1903301.}
\date{}						

\begin{document}

\begin{abstract}
	We prove prime geodesic theorems counting primitive closed geodesics on a compact hyperbolic 3-manifold with length and holonomy in prescribed intervals, which are allowed to shrink.
	Our results imply effective equidistribution of holonomy and have both the rate of shrinking and the strength of the error term fully symmetric in length and holonomy.
	
\end{abstract}

\subjclass[2010]{Primary 11F72; Secondary 53C22, 53C29, 58J50}
\keywords{Prime geodesic theorem, holonomy, hyperbolic 3-manifold, spectral geometry, trace formula}

\maketitle

\section{Introduction}

Closed geodesics on a smooth and connected Riemannian manifold $M$ act as important geometric and dynamical invariants. Closed geodesics support periodic orbits of the geodesic flow and in turn its invariant measures, 
while the length of the shortest closed geodesic of $M$ (its systole) acts as the first threshold of global geometry and dynamics. On locally symmetric spaces, the trace formula connects closed geodesics to the spectrum~of the Laplacian (which quantizes the geodesic flow), just as
elliptic elements of Fuchsian groups enter 
dimensions of spaces of cusp forms.
 In arithmetic cases, lengths and multiplicities of geodesics can often be explicitly related to invariants such as class numbers and regulators of indefinite binary quadratic forms. Hence for many reasons one seeks to understand the set of closed geodesics on $M$, and in particular, as in this paper, its size and structure.

\subsection{Prime geodesic theorems and holonomy}

We will be concerned with compact hyperbolic 3-manifolds $M$ arising as the quotient $M\simeq\Gamma\backslash\mathbb{H}^3$ of the hyperbolic upper half-3-space by a uniform, torsion-free lattice $\Gamma\subseteq G=\mathrm{PSL}_2\mathbb{C}$. Closed geodesics $C_{\gamma}$ on $M$ arise from non-identity conjugacy classes $[\gamma]$ in $\Gamma$, with primitive classes corresponding to infinitely many \emph{prime} geodesics of increasing lengths $\ell(\gamma)\to\infty$. The celebrated Prime Geodesic Theorem, in the form with an explicit error term, is due to Sarnak~\cite[Theorem 5.1]{Sarnak1983} and may be stated as
\begin{equation}
\label{PGT}
\pi_{\Gamma}(x):=\big|\big\{[\gamma]\text{ primitive in }\Gamma:\ell(\gamma)\leqslant x\big\}\big|=\Ei\nolimits_{\Gamma}(x)+\mathrm{O}_{\Gamma,\epsilon}\big(e^{(\frac53+\epsilon)x}\big),
\end{equation}
where the main term $\Ei_{\Gamma}(x)$ may be defined in terms of an absolutely continuous density measure $\dd\varpi_{\Gamma}$ as
\begin{equation}
\label{ei-varpi}
\Ei\nolimits_{\Gamma}(x)=\int_2^x \dd\varpi_{\Gamma}(t),\quad \dd\varpi_{\Gamma}(t)=\bigg(\frac{e^{2t}}t+\smash[t]{\sum_{j=1}^k}\frac{e^{(1+\nu_j)t}}t\bigg)\,\dd t,
\end{equation}
and $\{0<\nu_k\leqslant\dots\leqslant \nu_1<1\}$ correspond to the eigenvalues $1-\nu_j^2$ of the hyperbolic Laplacian $\Delta$ on $M$ in $(0,1)$ as described in \S\ref{background_sec}. The set of $\{\nu_j\}$ depends on $\Gamma$ only and is predicted by Selberg's conjecture to be empty for arithmetic $\Gamma$. In any case, $\Ei_{\Gamma}(x)\sim_{\Gamma}\Ei(2x)=\int_2^{2x} e^{t} /t \, \dd t\sim e^{2x}/2x$, recovering the general asymptotics of \cite[Theorems 3, 4]{Margulis1969} and \cite[Proposition 5.4]{GangolliWarner1980}. Moreover, \eqref{PGT} gives a power-saving asymptotic for $\pi_{\Gamma}(x,x+h):=\pi_{\Gamma}(x+h)-\pi_{\Gamma}(x)$ as long as $h\gg e^{-(1/3-\delta)x}$.

Attached to each closed geodesic $C_{\gamma}$ is the geometric action of the associated class $[\gamma]$. For
$\gamma\!\sim\!\left(\begin{smallmatrix} e^{(\ell+i\theta)/2}&0\\0&e^{-(\ell+i\theta)/2}\end{smallmatrix}\right)$ with $\ell>0$ (recall that $\Gamma$ is uniform), this is given by a shift by $\ell=\ell(\gamma)$ along $C_{\gamma}$ and a rotation around $C_{\gamma}$ by the angle $\theta$ (which corresponds to parallel transport along $C_{\gamma}$). Thus each $C_{\gamma}$ carries two geometric invariants: the length $\ell(\gamma)$ and the \emph{holonomy} $\hol(\gamma):= \theta \in\mathbb{R}/2\pi\mathbb{Z}$; in this setting of 3-manifolds one often also talks about the \emph{complex length} $\mathbb{C}\ell(\gamma):=\ell(\gamma)+i\hol(\gamma)$.
 It becomes a natural counting question to refine the count \eqref{PGT} according to holonomy. The equidistribution theorem of Sarnak--Wakayama~ \cite[Theorem~1, Corollary 1]{SarWa} in the present case of a compact 3-manifold $M\simeq\Gamma\backslash\mathbb{H}^3$ states that, for every interval $J\subseteq\mathbb{R}/2\pi\mathbb{Z}$,
\begin{equation}
\label{SarnakWakayamaResult}
\pi_{\Gamma}(x;J)=\big|\big\{[\gamma]\text{ primitive in }\Gamma:\ell(\gamma)\in [0,x],\,\hol(\gamma)\in J\big\}\big|\sim_{\Gamma}\frac{|J|}{2\pi}\pi_{\Gamma}(x)\quad (x\to\infty).
\end{equation}

We will be interested in asymptotics with precise control on $\hol(\gamma)$, the ``compact part of the complex length'', in the same way as the sharp-cutoff in $\ell(\gamma)$ in the Prime Geodesic Theorem \eqref{PGT}. In particular, our Theorem \ref{sharpsharp} proves the following effective version of the equidistribution result \eqref{SarnakWakayamaResult}:
\begin{equation}
\label{SW-effective}
\pi_{\Gamma}(x;J)=\frac{|J|}{2\pi}\pi_{\Gamma}(x)+\mathrm{O}_{\Gamma}\big(e^{5x/3}\big),
\end{equation}
with the implied constant \emph{independent of $J$}. We also remark that the error term in \eqref{SW-effective} is a pure exponential. Using dynamical methods, Margulis, Mohammadi, and Oh~\cite[Theorem 1.3]{MargulisMohammadiOh2014,MargulisMohammadiOh2014a} proved an asymptotic for $\pi_{\Gamma}(x;J)$ with a small, unspecified power savings in \eqref{SarnakWakayamaResult}, for a broad class of geometrically finite, Zariski dense $\Gamma<\PSL_2(\mathbb{C})$, including all lattices.

The asymptotic \eqref{SW-effective} provides power savings for the refined count with holonomy in an interval of length $|J|\gg e^{(-\frac13+\delta)x}$ and is reminiscent of a prime number theorem in arithmetic progressions with explicit level dependence. Indeed, we deduce it from the following uniform estimate on ``holonomy character sums'' (see Proposition~\ref{holonomy-sums-prop}):
\begin{equation}
\label{hol-char-sum-intro}
K_n(x):=\sum_{[\gamma]\text{ primitive in }\Gamma:{\ell(\gamma)\leqslant x}}e^{in\hol(\gamma)}\ll_{\Gamma} x^{-1} e^{5x/3} + n^2 e^x\quad(n\neq 0),
\end{equation}
where the implied constant depends only on $\Gamma$. The decomposition into pure harmonics of  holonomy in $\mathbb{R}/2\pi\mathbb{Z}$ implicit in the passage between \eqref{SW-effective} and \eqref{hol-char-sum-intro} (which was also the key implement in \cite{SarWa}) gives analytic access to automorphic constituents of $L^2(\Gamma\backslash G)$ via the trace formula; see \S\ref{spec-geom-sec}. Substantial cancellation in holonomy character sums in \eqref{hol-char-sum-intro}, which quickly leads to \eqref{SW-effective}, is in this sense analogous to classical bounds on sums of Dirichlet characters with explicit conductor dependence.

Theorem \ref{sharpsharp} requires fine control in both $\ell(\gamma)$ and $\hol(\gamma)$. More broadly, we argue that, for many purposes including counting in short ranges, the two geometric parameters $\ell(\gamma)$ and $\hol(\gamma)$ have the same standing, and that it is most natural to talk about the \emph{joint distribution} of the pair $(\ell(\gamma),\hol(\gamma))$. Such a result might be called an \emph{ambient prime geodesic theorem}. In Theorem~\ref{short-range-ambient-theorem}, we obtain our main result, an asymptotic count for primitive closed geodesics on $M$ according to the pair $(\ell(\gamma),\hol(\gamma))$.

\begin{theorem}\label{JointSmallIntervals}
	Let $\Gamma<\PSL_2\mathbb{C}$ be a discrete, co-compact, torsion-free subgroup. Then, for any intervals $I \subseteq [0,x]$ and $J\subseteq\mathbb{R}/2\pi\mathbb{Z}$,
	\begin{align*}
		\pi_{\Gamma}(I,J)
		&:=\big|\big\{[\gamma]\textnormal{ primitive in }\Gamma:(\ell(\gamma),\hol(\gamma))\in I\times J\big\}\big|\\
		&=\iint_{I\times J}\dd\varpi_{\Gamma}(t)\,\frac{\dd\theta}{2\pi}+\mathrm{O}_{\Gamma}\Big( (|I| + |J|)^{2/3} \frac{e^{5x/3}}{x^{2/3}} + \frac{e^{3x/2}}{x^{1/2}}\Big).
	\end{align*}
\end{theorem}

In particular, Theorem~\ref{JointSmallIntervals} proves the uniform asymptotic (see Theorem~\ref{sharpsharp} and its Corollary~\ref{ambient-short-range})
\begin{equation}
\label{Total-PGT}
\pi_{\Gamma}(I,J)
=\iint_{I\times J}\dd\varpi_{\Gamma}(t)\,\frac{\dd\theta}{2\pi}+\mathrm{O}_{\Gamma}\big(e^{5 x/3}\big),
\end{equation}
which provides a power savings as long as $|I\times J| \gg e^{(-\frac13+\delta)x}$ with $\delta>0$, where each interval may be short independently of each other; see Remark~\ref{shrinking-intervals-remark}. Theorem~\ref{JointSmallIntervals} further extends this range when both intervals $I$ and $J$ are short, down to as short as $|I\times J|\gg e^{(-\frac12+\delta)x}$ when $|I|\asymp |J|$; see Remark~\ref{concluding-remark} for details.

Joint distribution results such as our Theorem~\ref{JointSmallIntervals} may be seen as instances of spectral geometry on the group quotient $\Gamma\backslash G$, as we explain in \S\ref{spec-geom-sec}. We emphasize that our results apply regardless of whether the subgroup $\Gamma$ is arithmetic  or not; for asymptotics on arithmetic quotients in the length aspect, we refer to  \cite{SoundararajanYoung2013} as well as to recent advances on arithmetic hyperbolic 3-manifolds \cite{BalkanovaChatzakosCherubiniFrolenkovLaaksonen2019,BalkanovaFrolenkov2020a,BalogBiroCherubiniLaaksonen2020}. While it would certainly be of interest to obtain stronger error terms in ambient prime geodesic theorems on arithmetic hyperbolic 3-manifolds, our goal here 
is to establish universal, baseline results.

\subsection{Spectral geometry of \texorpdfstring{$\Gamma\backslash G$}{Gamma\textbackslash G}, trace formulas, and ambient counting} 
\label{spec-geom-sec}
Closed geodesics on $M$ are often considered along with the spectrum of the Laplace--Beltrami operator $\Delta_M$, which we recall involves averages over infinitesimal balls in $M$ and thus naturally quantizes the dynamics of the geodesic flow on $M$. On a rank one compact locally symmetric space $M=\Gamma\backslash S$, the classical spherical trace formula relates the eigenvalues of $\Delta_M$ (that is, frequencies of $\Delta_S$ appearing in $L^2(\Gamma\backslash S)$) with lengths of geodesics corresponding to non-identity conjugacy classes in $\Gamma$. This may also be  seen as the correspondence principle of quantum mechanics, relating long-term dynamics on $M$ with the semi-classical (high-energy) limit of the quantized system on $L^2(M)$, or as a non-commutative version of Fourier duality.

A key question of spectral geometry is whether isospectral manifolds (having the same Laplacian spectrum) are also isometric. For hyperbolic 3-manifolds, the answer is ``No'': in 1980, Vign\'eras found a pair of hyperbolic 3-manifolds that are isospectral but not isometric \cite{Vigneras}; see also \cite{Sunada, GarReid}. However,  Gangolli showed that for a compact hyperbolic 3-manifold, the Laplacian spectrum determines the set of lengths of closed geodesics \cite{Gangolli}. While for hyperbolic surfaces the Laplacian spectrum also determines the multiplicities of closed geodesics, in higher dimensions this is an open problem \cite[Remark 0.3]{Kelmer2014}. In the converse direction, Kelmer showed that the length spectrum (including multiplicities) determines the Laplacian spectrum for compact hyperbolic manifolds \cite{Kelmer2014}. 
 For arithmetic hyperbolic 3-manifolds, the complex length spectrum (and in fact, the rational length set) determines the commensurability class \cite{Reid, Chinbergetal}.
 
In any case, \emph{all structure} encoded in the trace formula should be reflected in the spectral-to-geometric correspondence. In the present case of $M\simeq\Gamma\backslash\PSL_2\mathbb{C}/\mathrm{PSU}_2$,
regular conjugacy classes $[\gamma]$ are parametrized by
$t_{\gamma}\in T/S_2=\{t=\big(\begin{smallmatrix} z&0\\0&z^{-1}\end{smallmatrix}\big):z\in\mathbb{C}^{\times}\}/(t\sim t^{-1})$,
$L^2_0(\Gamma\backslash\PSL_2\mathbb{C})$
is spanned by principal series representations $\pi=\mathrm{ind}_T^G(\chi_{\nu,p})$ indexed by unitary characters $\chi_{\nu,p}:T\to S^1$, and the trace formulas on
$\Gamma\backslash\PSL_2\mathbb{C}$
(Theorems~\ref{LL}, \ref{EvenTrace}, and \ref{OddTrace}, below) relate roughly,  up to fixed smooth weights,
\begin{equation}
\label{trace-formula-schematic}
\sum\nolimits_{\pi_{\nu,p}\subseteq L^2(\Gamma\backslash G)}\hat{F}(\chi_{\nu,p}^{-1})+
\dots\quad\leftrightsquigarrow\quad\sum\nolimits_{[\gamma]\subset\Gamma} F(t_{\gamma})+\dots,
\end{equation}
for a compactly supported smooth function $F:T/S_2\to\mathbb{C}$ and its Abel transform $\hat{F}:T^{\ast}/S_2\to\mathbb{C}$; see \S\ref{trace-formula-statement-sec} for details. Our guiding principle, then, is to view the spectral geometry of $\Gamma\backslash\mathbb{H}^3$ as duality between the classical spherical Maass forms and counts such as the classical Prime Geodesic Theorem \eqref{PGT}, essentially specializing \eqref{trace-formula-schematic} to $(\PSU_2 \cap T)$-invariant test functions, and the spectral geometry of $\Gamma\backslash G$ as 
the full duality, as encoded by \eqref{trace-formula-schematic}, between the entire spectrum of all $\PSU_2$-types and ambient prime geodesic theorems such as \eqref{SW-effective} and our Theorem~\ref{JointSmallIntervals}.

As an imperfect but convincing analogy, consider a higher rank real symmetric space, say a compact quotient $\Gamma\backslash G/K=\Gamma \backslash \SL_n\mathbb{R}/\SO_n$, with rank $r=n-1$, where $L^2_0(\Gamma\backslash G/K)$ decomposes into a direct sum of principal series representations induced from a character $\nu$ in the dual of the Lie algebra of the maximal torus $\mathfrak{a}^{\ast}\simeq\mathbb{R}^r$. Weyl's law is a classical spectral count, central for quantum mechanics or thermodynamics, of eigenforms in $L^2(\Gamma\backslash G/K)$ with  Laplace eigenvalue up to a given bound, or equivalently with $\|\nu\|_2\leqslant X$. From the point of view of equidistribution in families of automorphic forms ~\cite{Kowalski2013,SarnakShinTemplier2016,BrumleyMilicevic2018} one is more broadly interested in counting representations $\pi\subseteq L^2(\Gamma\backslash G/K)$  in a prescribed region $\Omega\subseteq\widehat{G}$ within the \emph{ambient} space $\widehat{G}\simeq\mathfrak{a}^{\ast}_{\mathbb{C}}$ of all such representations; a count with $r$ parameters rather than just one (see \cite[Proposition 7.2]{BrumleyMilicevic2018}).
The natural dual question to this is to count primitive conjugacy classes $[t_{\gamma}]$ occurring in $\Gamma$ not just according to their length $\|\log|t_{\gamma}|\|_2$
 but simply within prescribed regions 
in the $r$-dimensional space $[G]\simeq T/S_n$ of all regular conjugacy classes of $G$ (see \cite{Deitmar2004}).

 An ambient prime geodesic theorem in $\PSL_2\mathbb{C}$ such as our Theorem~\ref{JointSmallIntervals} should similarly count primitive conjugacy classes in $\Gamma$ according to their full parameter, the complex length $\mathbb{C}\ell(\gamma)\in\mathbb{C}/\{\pm 1\}$.  In a different context, Kelmer~\cite[Corollary 3.1]{Kelmer2012} proved effective equidistribution of holonomy for closed geodesics on products of $n$ hyperbolic planes (corresponding to conjugacy classes in an irreducible cocompact lattice $\Gamma<(\PSL_2\mathbb{R})^n$ that are hyperbolic at one place and elliptic at all others), including the asymptotic joint length-holonomy count.

This analogy brings about the natural question of using the non-spherical trace formula of \S\ref{trace-formula-statement-sec} to count geodesics in more general regions $\Omega\subseteq [G]$, the geometric counterpart to the spectral count of $\pi\subseteq L^2(\Gamma\backslash G)$ in regions of $\widehat{G}$ as in \cite[Proposition~7.2]{BrumleyMilicevic2018}. Passage to the geometric sharp-cutoff count and estimates of boundary terms as in \eqref{boundary-intro} in\-vol\-ve estimating contributions on the dual (spectral) side, extending deep into the tempered spectrum $\widehat{G}^{\mathrm{temp}}$ (in all directions), which is of more moderate growth. This in turn involves the
rate of decay of the Fourier transform $\widehat{\chi_{\Omega}}$ and thus the \emph{shape} of $\Omega$ and its boundary, 
as in classical lattice-point counting. We leave this intriguing lead for future work.

\subsection{Overview} For the sake of the reader, we now present an overview of the proof, omitting details. Throughout, $\Gamma < \PSL_2\mathbb{C}$ is a discrete, co-compact, torsion-free subgroup. We refer to \S\ref{background_sec} for background on the geometry and representation theory of $G=\PSL_2\mathbb{C}$ and its quotients.

One of the key tools we use is the non-spherical trace formula. We start with a version of Selberg's trace formula, explicated by Lin and Lipnowski~\cite[Corollary 2]{LinLip}, which captures both the length and holonomy of geodesics. We specialize this, on the spectral side, to representations of a particular type and, on the geometric side, to a particular frequency of holonomy, and obtain Theorem \ref{EvenTrace}, which states that for every $n\in\mathbb{Z}$ and every smooth, even, compactly-supported $g: \mathbb{R} \rightarrow \mathbb{C}$,
\begin{equation}
\label{tracenoutline}
\begin{aligned}
&\frac12\sum\limits_{\nu} (m_\Gamma(\pi_{\nu,n})+ m_\Gamma(\pi_{\nu,-n})) \int_{-\infty}^\infty g(u) e^{u \nu} \, \dd u
+\delta_0(n)\int_{-\infty}^{\infty}g(u)e^u\,\dd u-\frac{1}{2}\delta_{\pm 1}(n)\hat g(0)\\
&\qquad =  \frac{1}{2 \pi}\vol(\Gamma \backslash G)(n^2 g(0) - g''(0))+ \sum\limits_{[\gamma]} \ell(\gamma_0)  w(\gamma) g(\ell(\gamma)) \cos(n\hol(\gamma)),
\end{aligned}
\end{equation}
where the first sum is over the unitary principal and complementary series representations $\pi_{\nu,\pm n}$ (where $\nu\in i\mathbb{R}$ for principal series) occurring with multiplicities $m_{\Gamma}(\pi_{\nu,\pm n})$ in $L^2(\Gamma\backslash G)$, 
$w(\gamma) \asymp_{\Gamma}e^{-\ell(\gamma)}$, and the 
last sum is over the non-trivial hyperbolic and loxodromic conjugacy classes of $\Gamma$. For purposes of counting the length spectrum, where one is typically interested in \eqref{tracenoutline} with a test function $g$ of varying and extended support, it is natural to combine the complementary series and identity terms into a single integral of $g$ against a measure $\mathrm{d}\varpi_{\Gamma}^{\ast}$ related to $\mathrm{d}\varpi_{\Gamma}$; see \eqref{EvenTrace-alt}--\eqref{density-star-def}.
We also explicate a complementary ``odd'' trace formula, sampling geodesics with a weight $h(\ell(\gamma))\sin(n\hol(\gamma))$ for an odd $h\in C_c^{\infty}(\mathbb{R})$; see Theorem~\ref{OddTrace}.

In particular, using the trace formula~\eqref{tracenoutline} with a specific $g\in C_c^{\infty}(\mathbb{R})$ that emphasizes the spectral terms with $R-1 \leqslant |\nu| \leqslant R+1$ with weight $\hat g(i \nu/2 \pi)\gg 1$ and keeps all spectral terms non-negative,
we prove in Proposition ~\ref{Weyl} a bound on 
multiplicities of representations in an interval of fixed length:
\begin{equation}
\label{Weyl-intro}
\sum\limits_{R-1 \leqslant |\nu| \leqslant R+1} m_\Gamma(\pi_{\nu,n})
\ll \vol(\Gamma\backslash G)\cdot (R^2+n^2)+\mathrm{O}_{\Gamma}(1).
\end{equation}
The spectral bound \eqref{Weyl-intro} is a bound on local spectral densities, whose leading term agrees with the Plancherel measure. Such a bound is a standard tool in the passage from a smooth to sharp count of the spectrum in the proof of Weyl's law (cf.~\cite[Proposition~10.1]{BrumleyMilicevic2018}).

As the first step toward a sharp geodesic count, in Lemma~\ref{trace-estimates-lemma} we use the trace formula \eqref{tracenoutline} to get a handle on the geometric sum
\begin{equation}
\label{Tncos-intro}
T_n^{\cos}[g_{y,\eta}] = \sum_{[\gamma]} \ell(\gamma_0) w(\gamma) g_{y,\eta}(\ell(\gamma)) \cos(n\hol(\gamma))
\end{equation}
with a smooth, even function $g_{y,\eta}$ which approximates $\chi_{[-y,y]}$ and is supported on $[-y-\eta, y + \eta]$. The principal series terms in \eqref{tracenoutline} are then weighted with its Fourier transform $\hat{g}_{y,\eta}(i\nu/2\pi)$, which is essentially supported up to roughly $|\nu|\ll 1/\eta$ (exhibiting Schwartz decay past this range); we estimate these terms using \eqref{Weyl-intro}. In Lemma~\ref{geodesics-upper-bound}, we majorize the contributions to \eqref{Tncos-intro} of classes with $\ell(\gamma)\in 
[y-\eta,y+\eta]$ by a suitable non-negative smooth bump function on $[y-2\eta,y+2\eta]$ (whose Fourier transform again extends to roughly $\asymp 1/\eta$), and show with another application of \eqref{tracenoutline} and \eqref{Weyl-intro} that these boundary terms contribute $\mathrm{O}_{\Gamma}(e^y\eta+1/\eta^2)$ to \eqref{Tncos-intro}. Taking the boundary length $\eta=e^{-y/3}$, we obtain in Proposition ~\ref{trace-estimates-sharp}
\begin{equation}
\label{Tny-intro}
T_n^{\cos}(y) := \sum_{\ell(\gamma)\leqslant y}\ell(\gamma_0)w(\gamma)\cos(n\hol(\gamma))
=\delta_0(n)\int_{-y}^y\dd\varpi_{\Gamma}^{\ast}(u)+\mathrm{O}_{\Gamma}\big(e^{2y/3}+n^2y\big)
\end{equation}
and a corresponding asymptotic for $T_n^{\sin}(y)$ (and thus $T_n(y)$ in \eqref{TnTnP}).

Counts such as \eqref{Tny-intro} arising from the trace formula naturally involve the weights $w(\gamma)$ (from the Weyl discriminant) and all conjugacy classes $[\gamma]$, including imprimitive ones. In Section~\ref{PreliminariesSection} we address these two technical aspects and show they are essentially harmless, namely, using that $w(\gamma)=e^{-\ell(\gamma)}+\mathrm{O}_{\Gamma}(e^{-2\ell(\gamma)})$ and that imprimitive classes contribute comparatively very few (namely $\mathrm{O}_{\Gamma}(e^y)$) terms to \eqref{Tny-intro}, we prove in Lemma~\ref{weights-primitivity} that a simpler sum $S_n^P(y)$ over primitive geodesics satisfies
 \begin{equation}
 \label{SnP-intro}
 S_n^{P}(y):= \sump_{\ell(\gamma) \leqslant y} \ell(\gamma) e^{-\ell(\gamma)+i n \hol \gamma}=T_n(y)+\mathrm{O}_{\Gamma}(y).
 \end{equation}
Combining \eqref{Tny-intro} and \eqref{SnP-intro} and using integration by parts, we obtain in Proposition~\ref{holonomy-sums-prop}, as cited above, the estimate \eqref{hol-char-sum-intro} on ``holonomy character sums'' $K_n(y)$. In particular, Proposition~\ref{holonomy-sums-prop} for $n\neq 0$ quantifies cancellation among the holonomies of primitive geodesics with $\ell(\gamma)\leqslant y$ and shows that these are equidistributed throughout $\mathbb{R}/2\pi\mathbb{Z}$.

While we have so far in \eqref{Tny-intro}, \eqref{SnP-intro}, and \eqref{hol-char-sum-intro} emphasized the traditional sharp cutoff $\ell(\gamma)\leqslant y$, we in fact throughout also prove estimates for analogous smooth-cutoff quantities like $S_n^P[g_{y,\eta}]$ and $K_n[g_{y,\eta}]$, with explicit dependence on $\eta$. Moreover, we observe structural analogies in the length and holonomy aspects, such as in the comparison of our Lemmata~\ref{geodesics-upper-bound} and \ref{passage-to-sharp-holonomy}, which state roughly that
\begin{equation}
\label{boundary-intro}
\sum_{y-\eta\leqslant\ell(\gamma)\leqslant y+\eta}\ell(\gamma_0)w(\gamma)\ll_{\Gamma} \eta e^y+\frac1{\eta^2},\quad
\sum_{\substack{\ell(\gamma)\leqslant y\\ \theta_0-\eta'\leqslant\hol(\gamma)\leqslant\theta_0+\eta'}}\ell(\gamma_0)w(\gamma)\ll_{\Gamma} \eta' e^y+\frac y{\eta'{}^2}.
\end{equation}
This analogy is fundamentally due to the fact that, in each case, the dual (spectral) sum over $\pi_{\nu,n}$ extends up to roughly $1/\eta$ or $1/\eta'$ (in the $\nu$- and $n$-direction, respectively) and that the Plancherel measure shown in \eqref{Weyl-intro} and \eqref{plancherel} is symmetric; see Remark~\ref{symmetry-remark}.

With this in mind, in \S\ref{smooth-count-subsec} and \S\ref{ambient-pgt-sharp-section} we prove $4=2\times 2$ asymptotic formulas for ``ambient'' prime geodesic counts, beginning with the smooth count (Proposition~\ref{SmoothSmooth}) of the form
\begin{equation}
\label{smooth-smooth-overview}
\begin{aligned}
\pi_{\Gamma}(g_{y,\eta},f)
&:=\sump_{[\gamma]} f(\hol(\gamma)) g_{y,\eta}(\ell(\gamma))\\
&= \frac1{2\pi}\int_0^{2\pi}f(\theta)\,\dd\theta\cdot\int_2^{\infty}g_{y,\eta}(u)\,\dd\varpi_{\Gamma}(u)+ \mathrm{O}_{\Gamma, \eta_0}\Big(\frac{e^y}{y \eta^2} \| \hat f\|_1 + e^y \|\hat f\|_{2,1}\Big)
\end{aligned}
\end{equation}
for a smooth $f:\mathbb{R}/2\pi\mathbb{Z}\to\mathbb{C}$, $\| \hat f\|_{2,1} := \|\hat f\|_1 + \|\widehat{f''}\|_1$, and $0 < \eta \leqslant \eta_0$, and then for the related counts $\pi_{\Gamma}(y,f)$, $\pi_{\Gamma}(g_{y,\eta},J)$, and $\pi_{\Gamma}(y,J)$ (for $y>0$ and any interval $J\subseteq\mathbb{R}/2\pi\mathbb{Z}$), which have sharp cutoffs in the length, holonomy, and in both aspects, respectively. In particular, by spectrally expanding $f$ into a Fourier series and estimating $K_n(y)$ using \eqref{hol-char-sum-intro}, in Theorem~\ref{SarnakTheorem} we prove that
\begin{equation*}
\pi_{\Gamma}(y,f):=\sump_{\ell(\gamma) \leqslant y} f(\hol(\gamma))
=\frac1{2\pi}\int_0^{2\pi}f(\theta)\,\dd\theta\cdot\int_2^y\dd\varpi_{\Gamma}(u)+\mathrm{O}_{\Gamma}\Big(\|\hat{f}\|_1\frac{e^{5y/3}}y+\|\widehat{f''}\|_1e^y\Big),
\end{equation*}
which recovers \cite[Theorem~1]{SarWa} for a fixed $f$. However, our explicit dependence on $f$ allows us to choose a smooth $f_{J, \eta'}$ approximating a sharp holonomy cutoff, while maintaining explicit dependence on $\eta'$, and then estimate the terms with holonomy within $\eta'$ of the boundary of $J$ using the second bound in \eqref{boundary-intro}. In exact analogy with the passage to the sharp count in \eqref{Tny-intro}, we choose $\eta'=e^{-y/3}$, which leads in Theorem~\ref{sharpsharp} to the effective count \eqref{SW-effective} for a sharp length and holonomy count, and then to its consequence \eqref{Total-PGT}.

In \S\ref{APGT-shrinking-subsec}, we prove asymptotic formulas that provide counts for the number of geodesics in intervals $I$ and $J$ of length and holonomy, respectively. These counts feature a combination of sharp and smooth cutoffs in the length and holonomy, in complete parallel to \S\ref{ambient-pgt-sharp-section}, but with improvements in the error term when the lengths of $I$ and/or $J$ are shrinking.
After explicating in Lemma~\ref{smooth-smooth-intervals} an asymptotic analogous to \eqref{smooth-smooth-overview} for the count $\pi_{\Gamma}(g_{I,\eta},f_{J,\eta'})$ with suitable smooth length and holonomy cutoffs, we derive as Corollary~\ref{short-range-upper-bound} the upper bound
\begin{equation}
\label{ambient-boundary-intro}
\sum_{\substack{\ell(\gamma)\in I\\\hol(\gamma)\in J}} 1
\ll_{\Gamma, \eta_0} (|I|+\eta)(|J|+\eta')\frac{e^{2y}}{y} + \frac{e^y}{y \eta^2} \log^{\ast}\frac{1}{\eta'} + \frac{e^y}{\eta'^2},
\end{equation}
for every  $I \subseteq [0,y]$, $J \subseteq \mathbb{R}/2 \pi \mathbb{Z}$ and $0<\eta\leqslant\eta_0$, $0 <\eta'\leqslant 2 \pi$.
This is a (normalized; recall that $w(\gamma)\sim_{\Gamma}e^{-\ell(\gamma)}$) ambiental analogue of \eqref{boundary-intro}, where we additionally profit in the first term when lengths and holonomies are sampled from short intervals. Then we estimate the smooth count using Lemma \ref{smooth-smooth-intervals} and bound the ambiguous regions with (\ref{ambient-boundary-intro}), which leads to Proposition \ref{sharp-length-intervals-prop} and our main result, Theorem~\ref{JointSmallIntervals}.

\subsection{Notation}
\label{notation-subsec}
We write $f=\mathrm{O}(g)$ or $f\ll g$ to mean that $|f|\leqslant Cg$ for some constant $C>0$, which may be different from line to line and is absolute unless explicitly indicated with a subscript. We also write $f\asymp g$ to denote that $f\ll g\ll f$, and $f\sim g$ to denote that $\lim f/g=1$, where the direction of the limit is clear from the context, again with dependencies of implied constants and rate of convergence only as indicated. (No confusion should arise with the usage of $\sim$ to also denote conjugate elements in a matrix group.)

We use the convention $\hat f(\xi) = \int_{\mathbb{R}} f(x) e^{-2 \pi i x \xi} \, \dd \xi$ to denote the Fourier transform of a Schwartz function $f$, and we choose the normalization $\hat f(n) = \int_{\mathbb{R}/2 \pi \mathbb{Z}} f(x)e^{-inx} \, \dd x$ for the Fourier coefficient of a periodic function $f$ on $\mathbb{R}/2 \pi \mathbb{Z}$. We also use the shorthand notation $\log^{\ast}(x) := \log(x + 2)$ for $x>0$. Finally, we write $\delta_a(n)$ for the Kronecker delta-function, that is $\delta_a(n)=1$ if $n=a$ and 0 otherwise (including writing $\delta_{\pm 1}(n)$ as shorthand for whether $n=\pm 1$ or not).

\subsection{Acknowledgements} We are very grateful to the anonymous referees for their careful reading and thoughtful comments, which have significantly improved the manuscript, including a crucial observation that led to the improvement in our ambient counts over short ranges. The second author would like to thank the Max Planck Institute for Mathematics for their support and outstanding working conditions and research atmosphere.

\section{Non-Spherical Trace Formulas and Weyl's Law}

\subsection{Background on groups and representations}
\label{background_sec}

Let $G =\PSL_2\mathbb{C}$, and let $\Gamma$ be a discrete, torsion-free, co-compact subgroup of $G$. The group $G = \PSL_2\mathbb{C}$ is in one-to-one-correspondence with the group of orientation-preserving isometries of the 3-dimensional, hyperbolic upper half space $\mathbb{H}^3$, which we describe shortly. We are primarily concerned with the geometry of the fundamental domain $M=\Gamma\backslash\mathbb{H}^3$, which is a compact hyperbolic 3-manifold, 
and its covering $\Gamma\backslash G$. In this section, we collect some background material about the group $G$, its geometric action on $\mathbb{H}^3$, the geometry of $M$, and the representation theory of $L^2(G)$ and $L^2(\Gamma\backslash G)$.

The group $\PSL_2\mathbb{C}$ has the Iwasawa decomposition $G=UAK$, where
\begin{gather*}
 U = \left\{ \begin{pmatrix}
1 & z\\
0 & 1
\end{pmatrix}: z = x+iy \in \mathbb{C},\,x,y\in\mathbb{R}\right\}, \quad
A = \left\{\begin{pmatrix}
e^{u/2} & 0\\
0 & e^{-u/2}
\end{pmatrix} : u \in \mathbb{R}\right\},\\ 
K = 
\PSU_2=\left\{\begin{pmatrix}
\alpha & \beta\\
- \bar \beta & \bar \alpha
\end{pmatrix} : \alpha, \beta \in \mathbb{C},\, |\alpha|^2 + |\beta|^2 = 1\right\}/\{\pm 1\}.
\end{gather*}
Here, $UA$ is a Borel subgroup of $G$ with the unipotent subgroup $U$ and with $A$ a maximal torus in $G$, and $K$ is a maximal compact subgroup of $G$. The Haar measure on each of these subgroups is unique up to a constant multiple. We choose the Euclidean measure $\dd x\, \dd y$ on $U$, $\dd u$ on $A$, and the volume $1$ Haar measure $\dd k$ on $K$. This induces a Haar measure on $G$.

The quotient $G/K$ may be identified with the upper half space $\mathbb{H}^3=\{z+ir:z\in\mathbb{C},\,r>0\}$, a 3-dimensional hyperbolic space with the hyperbolic metric $\dd s=(|\dd z|^2+(\dd r)^2)^{1/2}/r$. The action of $G$ by left multiplication induces an action of $G$ on $\mathbb{H}^3$ by orientation-preserving isometries (preserving the hyperbolic metric), which may also be described in terms of $(z,r)$-coordinates; see \cite[\S1.1]{EGM}. In fact, the group $G$ is in one-to-one-correspondence with the group of orientation-preserving isometries of $\mathbb{H}^3$.

Elements of $G$ can be classified into identity, parabolic, elliptic, hyperbolic, and loxodromic, each with a distinct type of geometric action on $\mathbb{H}^3$. Parabolic elements are conjugate to an element of the unipotent group $U$ described above. All other elements are diagonalizable. Every diagonalizable element $\gamma \in G$ is conjugate to some 
\[t_\gamma\in T:=\left\{\begin{pmatrix}
e^{(u + i \theta)/2} & 0 \\
0 & e^{-(u + i \theta)/2}
\end{pmatrix}:u\in\mathbb{R},\,\theta\in\mathbb{R}/2\pi\mathbb{Z}\right\}.\]
We will also refer to a matrix of this form as $t_{u,\theta}$. If $u = 0$, then $\gamma$ is elliptic; if $\theta = 0$, then $\gamma$ is hyperbolic. All other elements are said to be loxodromic, and this term is sometimes also applied to hyperbolic elements. Note that we can conjugate $t_\gamma$ by 
$\big(\begin{smallmatrix} 0 & -1\\ 1 & 0\end{smallmatrix}\big)$
to swap the diagonal elements; therefore, we can and do choose the length/holonomy pair to have non-negative length.

If $\gamma$ is hyperbolic or loxodromic, it has two fixed points on the boundary $\partial\mathbb{H}^3\cup\{\infty\}=\hat{\mathbb{C}}$ and acts on $\mathbb{H}^3$ by a shift along the geodesic connecting these two fixed points by the \emph{length} $\ell(\gamma) = u$ followed by a rotation around the same axis by the \emph{holonomy} $\hol(\gamma) = \theta$. We also talk about the \emph{complex length} $\mathbb{C} \ell(\gamma) = u + i \theta$. A hyperbolic or loxodromic $\gamma\in\Gamma$ corresponds to a closed geodesic in $\Gamma \backslash \mathbb{H}^3$. Since $\Gamma$ is discrete, $\Gamma \backslash \mathbb{H}^3$ contains a geodesic of minimum length, which we refer to throughout the paper as $\eta_0(\Gamma)$.
 
The group $T$ of diagonal elements of $G$ has a Haar measure $\dd u \, \dd\theta/{2\pi}$. For $\nu\in\mathbb{C}$ and $p\in\mathbb{Z}$, define the character $\chi_{\nu,p}$ on $T$ by
\begin{equation}
\label{characters-T}
\chi_{\nu,p}\left(\begin{pmatrix}
 e^{(u + i \theta)/2} & 0 \\
 0 & e^{-(u + i \theta)/2}
 \end{pmatrix}\right) = e^{u \nu + i p \theta}.
 \end{equation}
This is a unitary character for $\nu\in i\mathbb{R}$.

The classification of irreducible, unitary representations of $G = \PSL_2 \mathbb{C}$ is classical. Let $\pi_{\nu,p}$ denote the representation of $G$ obtained by extending the character $\chi_{\nu,p}$ to $B=UT$ trivially along $U$ and then inducing unitarily to $G$. The unitary irreducible representations of $G$ are then as follows:
\begin{itemize}
\item the trivial representation;
\item the unitary principal series representations $\pi_{\nu,p}$, for $\nu \in i \mathbb{R}$ and $p \in \mathbb{Z}$;
\item the complementary series representations $\pi_{\nu,0}$, for $0 <\nu <1$.
\end{itemize} 
The only equivalences among the above irreducible representations of $G$ are that $\pi_{\nu, p} \cong \pi_{- \nu, -p}$ \cite[Theorem 16.2]{Knapp}. For a complete description of the principal and complementary series representations including the $G$-invariant inner product, see \cite[\S II.4]{Knapp}, \cite[\S 2.4.1]{LinLip}. 

 Let $L^2(\Gamma \backslash G)$ be the space of square-integrable functions on $\Gamma \backslash G$. The group $G$ acts on $L^2(\Gamma \backslash G)$ by the right-regular representation. Then we have the decomposition
 \begin{equation}
 \label{L2decomp}
 L^2(\Gamma \backslash G) = \bigoplus\nolimits_{\pi \in \hat G} m_\Gamma(\pi) \pi,
 \end{equation}
 where $\hat G$ is the set of irreducible, unitary representations of $G$ and the non-negative integer $m_\Gamma(\pi)$ is the multiplicity of $\pi$ in $L^2(\Gamma \backslash G)$. Since $\Gamma$ is co-compact, the non-vanishing terms in \eqref{L2decomp} form a countable sum that may (after the trivial representation) be double-indexed by $\pi_{pj}\simeq\pi_{\nu_{pj},p}$, where, for every $p\in\mathbb{Z}$, $|\nu_{pj}|\to\infty$ ($j\to\infty$). 

Decomposition \eqref{L2decomp} into irreducible representations is fundamentally connected to the theory of automorphic forms. Each representation $\pi$ appearing in $L^2(\Gamma \backslash G)$ corresponds to an irreducible representation space $V_\pi$. The Casimir element of $G$ acts on $V_\pi$ by scalar multiplication, and $V_\pi$ is spanned by $\Gamma$-automorphic functions. For an explicit description, see \cite[Chapter 8]{Lok}.

\subsection{Non-Spherical Trace Formulas}
\label{trace-formula-statement-sec}

For a co-compact discrete subgroup $\Gamma<G$, the Selberg trace formula relates spectral information about the multiplicities of representations in $L^2(\Gamma \backslash G)$ to geometric information about elements of $\Gamma$. This formula results from computing the trace of the resolvent operator in two ways.

When the underlying kernel is bi-$K$-invariant, this recovers the classical Selberg trace formula on the compact hyperbolic manifold $M=\Gamma\backslash G/K$. For example, in the present rank one case $G=\PSL_2\mathbb{C}$, the trace formula~\cite[Theorem~5.1]{EGM} (after removing the Eisenstein, parabolic, and elliptic terms) relates the spectrum of the Laplacian on $L^2(M)$ with the lengths of closed geodesics on $M$; see also \cite[Theorem 10.2]{Iwaniec1995} for the more familiar case $G=\SL_2\mathbb{R}$.

Full control over the holonomy of geodesics on $M$ requires a trace formula on $\Gamma\backslash G$. The following trace formula was explicated by Lin and Lipnowski for compact, hyperbolic 3-manifolds. We refer to \S\ref{background_sec} for notations.

\begin{theorem}[Lin--Lipnowski~{\cite[Corollary 2]{LinLip}}]
\label{LL} Let $\Gamma<\PSL_2\mathbb{C}$ be a discrete, co-compact, torsion-free subgroup. Then, for every smooth, compactly-supported function $F:T\to\mathbb{C}$ such that $F(t) = F(t^{-1})$,	
	\begin{align*}
	&\sum\limits_{\nu,p} m_\Gamma(\pi_{\nu,p}) \hat{F}(\chi_{\nu,p}^{-1}) + \frac{1}{2} \int_T |D(t^{-1})|^{1/2} F(t) \, \dd t\\
	&\qquad= - \frac{1}{2\pi} \vol(\Gamma \backslash G) \Big(\p{^2}{u^2} + \p{^2}{\theta^2}\Big)F\Big|_{t = 1} + \sum\limits_{[\gamma]} \ell(\gamma_0) |D(t_\gamma^{-1})|^{-1/2} F(t_\gamma),
	\end{align*}
where $\hat F(\chi) = \int_T F(t) \chi^{-1}(t) \,\dd t$ is the Abel transform and $D(t_\gamma) = (1- e^{\mathbb{C} \ell(\gamma)})^2(1 - e^{- \mathbb{C} \ell(\gamma)})^2$ is the Weyl discriminant. The first sum is over unitary principal and complementary series representations $\pi_{\nu,p}$, and $m_\Gamma(\pi_{\nu,p})$ refers to the multiplicity of a representation $\pi_{\nu, p}$ in $L^2(\Gamma \backslash G)$. The latter sum is over the non-trivial conjugacy classes $[\gamma]$ of $\Gamma$, and $\ell(\gamma_0)$ refers to the length of the geodesic corresponding to the primitive element $\gamma_0$ that generates $\gamma$.
\end{theorem}

Strictly speaking, there are two elements, $\gamma_0$ and $\gamma_0^{-1}$, which generate $\gamma$; however, $\ell(\gamma_0) = \ell(\gamma_0^{-1})$ and so here and henceforth we ignore this distinction.

Lin and Lipnowski used Theorem~\ref{LL} with $F(t_{u,\theta}) = g(u) \cos(\theta)$, where $g$ is even, smooth, and compactly supported, and evaluated the above equation to isolate the representations $\pi_{\nu, \pm 1}$, which in turn gives a handle on the first eigenvalue of the Hodge Laplacian acting on coexact 1-forms. From a more analytic perspective, irreducible representations of $K=\PSU_2$ are classified as $(2\ell+1)$-dimensional representations $\tau_{\ell}$ ($\ell\geqslant 0$), and according to the right $K$-action we have the decompositions $\pi_{\nu,p}\vert_K\simeq\bigoplus_{\ell=|p|}^{\infty}\tau_{\ell}$ and
\begin{equation}
\label{decomp_p}
L^2(\Gamma\backslash G)=\bigoplus\nolimits_{\ell=0}^{\infty}L^2(\Gamma\backslash G)_{\ell}.
\end{equation}
In the following two theorems, we similarly specialize Theorem~\ref{LL} to isolate multiplicities of representations $\pi_{\nu,\pm p}\subset L^2(\Gamma\backslash G)$ whose lowest $K$-weight vectors occur in a fixed component $L^2(\Gamma\backslash G)_{|p|}$. Theorems~\ref{EvenTrace} (which we adapt from~\cite{Dever2019}) and \ref{OddTrace} are the ``even'' and ``odd'' trace formulas and should be compared to~\cite[Theorem~6.5]{SarWa}. Indeed, the intrinsic symmetry in Theorem~\ref{LL} imposes two equalities among the four quantities $m_{\Gamma}(\pi_{\pm\nu,\pm p})$, so that two trace formulas provide for the fullest possible spectral resolution in \eqref{decomp_p}. Note that $p=0$ corresponds to the familiar spherical Maass forms on
$\Gamma\backslash\mathbb{H}^3$, 
 in which case Theorem~\ref{EvenTrace} recovers the classical spherical trace formula~\cite[Theorem~5.1]{EGM} for compact hyperbolic 3-manifolds.

\begin{theorem}\label{EvenTrace} Let $\Gamma<\PSL_2\mathbb{C}$ be a discrete, co-compact, torsion-free subgroup, and let $n\in\mathbb{Z}$. Then, for every smooth, even, compactly supported function $g: \mathbb{R}\to\mathbb{C}$,
\begin{equation}
\label{tracen}
\begin{aligned}
&\frac12\sum\limits_{\nu} (m_\Gamma(\pi_{\nu,n})+ m_\Gamma(\pi_{\nu,-n})) \int_{-\infty}^\infty g(u) e^{u \nu} \, \dd u\\
&\qquad\quad+\delta_0(n)\int_{-\infty}^{\infty}g(u)\ e^u \,\dd u-\frac{1}{2}\delta_{\pm 1}(n)\hat g(0)\\
	& =  \frac{1}{2 \pi}\vol(\Gamma \backslash G)(n^2 g(0) - g''(0))\\
	&\qquad\quad+ \sum\limits_{[\gamma]} \ell(\gamma_0) |1-e^{\mathbb{C} \ell(\gamma)}|^{-1}  |1-e^{-\mathbb{C} \ell(\gamma)}|^{-1} g(\ell(\gamma)) \cos(n\hol(\gamma)).
	\end{aligned}
\end{equation}
The first sum is over unitary principal (for $n=0$, also complementary) series representations $\pi_{\nu,\pm n}$, $m_\Gamma(\pi_{\nu,\pm n})$ refers to the multiplicity of a representation $\pi_{\nu, \pm n}$ in $L^2(\Gamma \backslash G)$, and $\delta_0$, $\delta_{\pm 1}$ are
as in \S\ref{notation-subsec}.
The latter sum is over the non-trivial hyperbolic and loxodromic conjugacy classes $[\gamma]$ of $\Gamma$, and $\ell(\gamma_0)$ refers to the length of the geodesic corresponding to the primitive element $\gamma_0$ which generates $\gamma$.
\end{theorem}

\begin{remark}
The left-hand side in Theorem~\ref{EvenTrace} may be rewritten as
\begin{equation}
\label{EvenTrace-alt}
\begin{aligned}
&\delta_0(n)\int_{-\infty}^{\infty}g(u)\,\dd{\varpi}^{\ast}_{\Gamma}(u)\\
&\qquad+\frac12\sum\limits_{\nu\in i\mathbb{R}} (m_\Gamma(\pi_{\nu,n})+ m_\Gamma(\pi_{\nu,-n})) \hat{g}\Big(\frac{i\nu}{2\pi}\Big)-\frac{1}{2}\delta_{\pm 1}(n)\hat g(0),
\end{aligned}
\end{equation}
where the sum is now only over unitary principal series representations $\pi_{\nu,\pm n}$ with $\nu\in i\mathbb{R}$ (including for $n=0$), and ${\varpi}^{\ast}_{\Gamma}$ is the absolutely continuous measure on $\mathbb{R}$ given by
\begin{equation}
\label{density-star-def}
\dd{\varpi}^{\ast}_{\Gamma}(u)=\bigg(e^u+\sum_{\nu\in(0,1)}m_{\Gamma}(\pi_{\nu,0})e^{u\nu}\bigg)\,\dd u,
\end{equation}
with the latter sum being over the complementary spectrum $\pi_{\nu,0}$ occurring in $L^2(\Gamma\backslash G)$. The form \eqref{EvenTrace-alt} is particularly well suited to geodesic counting, with $\varpi_{\Gamma}^{\ast}$ acting as the density of the length spectrum $[\gamma]$ of $\Gamma$ (not necessarily primitive, and weighted by $\ell(\gamma_0)w(\gamma)$ with $w(\gamma)\asymp_{\Gamma} e^{-\ell(\gamma)}$ as in \eqref{tracen} and \eqref{w-gamma-def}).
\end{remark}

\begin{proof}
Consider the function $F_n:T\to\mathbb{C}$ defined as $F_n (t_{u,\theta}) = g(u) \cos(n \theta)$. This is a smooth, compactly supported function on $T$ invariant under inverses. Therefore, Theorem~\ref{LL} applies to $F_n$; we will explicate each term.
	
On the spectral side,
\begin{align*}
	\hat F_{n}(\chi_{\nu,p}^{-1}) 
	&=\frac{1}{4\pi} \int_{-\infty}^\infty \int_0^{2 \pi} g(u) \big(e^{u\nu+ i(p+n)\theta} + e^{u\nu + i(p-n) \theta}\big) \,\dd\theta\,\dd u\\
	&=\frac{1}{4\pi} \int_{-\infty}^\infty g(u) e^{u\nu} \bigg( \int_0^{2\pi} \big(e^{i(p+n)\theta} + e^{i(p-n)\theta}\big)\,\dd\theta \bigg)\,\dd u.
	\end{align*}
Therefore, for $n \neq 0$, $\hat F_n(\chi_{\nu, \pm n}^{-1}) = \frac{1}{2} \int_{-\infty}^\infty g(u) e^{u \nu} \,\dd u$, and $\hat F_0(\chi_{\nu,0}^{-1}) = \int_{-\infty}^\infty g(u) e^{u \nu} \,\dd u$. If $p \neq \pm n$, then  $\hat F_n(\chi_{\nu,p}^{-1}) = 0$. The contribution of the trivial representation is
\begin{equation}
\label{TDu-theta}
\begin{aligned}
	\frac{1}{2}\int_T |D(t_{u,\theta}^{-1})|^{1/2} F_n(t_{u,\theta}) \,\dd t_{u,\theta}
	&= \frac{1}{4\pi} \int_{-\infty}^\infty \int_0^{2\pi} |e^{u+ i \theta}|\big |1- e^{-(u+i\theta)}\big|^2 g(u) \cos(n\theta) \, \dd\theta \,\dd u\\
	&= \frac{1}{4\pi} \int_{-\infty}^\infty \int_0^{2\pi} (e^u + e^{-u}- 2 \cos \theta)g(u) \cos(n\theta) \,\dd\theta \, \dd u\\
	 &=\delta_0(n)\int_{-\infty}^{\infty}g(u)\cosh u\,\dd u-\frac12\delta_{\pm 1}(n)\int_{-\infty}^{\infty}g(u)\,\dd u,
	\end{aligned}
	\end{equation}
by orthogonality. Note that $\int_{-\infty}^\infty g(u) \cosh u \, \dd u = \int_{-\infty}^\infty g(u) e^{u} \, \dd u$ since $g$ is even.
	
On the geometric side, the contribution of the non-trivial hyperbolic and loxodromic elements is
\[\ell(\gamma_0)|D(t_\gamma^{-1})|^{-1/2} F_n(t_\gamma) = \ell(\gamma_0) |1-e^{\mathbb{C} \ell(\gamma)}|^{-1}  |1-e^{-\mathbb{C} \ell(\gamma)}|^{-1} g(\ell(\gamma)) \cos(n\hol(\gamma)).\]
For the contribution of the identity element on the geometric side, we have
\[- \frac{1}{2 \pi}\vol(\Gamma \backslash G) \Big(\p{^2}{u^2} + \p{^2}{\theta^2}\Big)g(u) \cos(n\theta)\Big|_{u = 0, \theta = 0} = \frac{1}{2 \pi}\vol(\Gamma \backslash G)(n^2 g(0) - g''(0)). \qedhere\]
\end{proof}

Now that we have a formula for the sum of multiplicities $m_{\Gamma}(\pi_{\nu,n})+m_{\Gamma}(\pi_{\nu,-n})$, we also require an understanding of the difference in these multiplicities (both subject to the symmetry $\pi_{\nu,n}\cong\pi_{-\nu,-n}$). This can be achieved by capturing both the length and the holonomy with an odd function.

\begin{theorem}\label{OddTrace} Let $\Gamma<\PSL_2\mathbb{C}$ be a discrete, co-compact, torsion-free subgroup, and let $n\in\mathbb{Z}$. Then, for every smooth, odd, compactly supported function $h:\mathbb{R}\to\mathbb{C}$,
\begin{align*}
&\frac{1}{2} i \sum_{\nu} (m_{\Gamma}(\pi_{\nu,n}) - m_\Gamma(\pi_{\nu,-n})) \int_{-\infty}^\infty h(u) e^{u \nu} \,\dd u\\
	&\qquad = \sum\limits_{[\gamma]} \ell(\gamma_0) |1-e^{\mathbb{C} \ell(\gamma)}|^{-1}  |1-e^{-\mathbb{C} \ell(\gamma)}|^{-1} h(\ell(\gamma)) \sin(n\hol(\gamma)),
\end{align*}
where the terms are defined as in Theorem~\ref{EvenTrace}.

\end{theorem}
 
 \begin{proof}
This time we consider the function $H_n:T\to\mathbb{C}$ given by $H_n (t_{u,\theta}) = h(u) \sin(n \theta)$. This is a smooth, compactly supported function on $T$ invariant under inverses. Therefore, Theorem~\ref{LL} applies to $H_n$, and we explicate all terms.
As for the spectral terms, we compute
 	\begin{align*}
 	\hat H_n(\chi_{\nu,p}^{-1}) &= \frac{1}{2\pi} \int_{-\infty}^\infty \int_0^{2 \pi} H_n(t_{u,\theta}) e^{u\nu+ ip \theta}\,\dd\theta\,\dd u\\
	& = \frac{1}{2 \pi} \hat h \Big(\frac{i \nu}{2\pi}\Big) \int_0^{2 \pi} \frac{1}{2i} \big(e^{i (p + n) \theta} - e^{i (p - n)\theta}\big)\,\dd \theta.
\end{align*}
By orthogonality, this vanishes unless 
$p=\pm n\neq 0$, in which case
\[ \hat{H}_n(\chi_{\nu,n}^{-1})=\frac12i\hat{h}\Big(\frac{i\nu}{2\pi}\Big),\quad\hat{H}_n(\chi_{\nu,-n}^{-1})=-\frac12i\hat{h}\Big(\frac{i\nu}{2\pi}\Big). \]
Thus the spectral terms contribute $\frac{1}{2} i \sum_{\nu} (m_{\Gamma}(\pi_{\nu,n}) - m_\Gamma(\pi_{\nu,-n})) \hat h({i \nu}/{2 \pi})$, as advertised.
	
The terms corresponding to the trivial representation and the identity element vanish. Indeed, using the evaluation \eqref{TDu-theta} from the proof of Theorem~\ref{EvenTrace} (which in particular shows that $|D(t_{u,\theta}^{-1})|^{1/2}$ is an even function of $\theta$),
\[ \frac{1}{2} \int_T |D(t_{u,\theta}^{-1})|^{1/2} H_n(t_{u,\theta}) \,\dd t_{u,\theta} = \frac{1}{4\pi} \int_{-\infty}^\infty \int_0^{2\pi} (e^u + e^{-u}- 2 \cos \theta)h(u) \sin(n\theta) \, \dd \theta \, \dd u=0, \]
while
\[- \frac{1}{2 \pi}\vol(\Gamma \backslash G) \Big(\p{^2}{u^2} + \p{^2}{\theta^2}\Big)h(u) \sin(n\theta)\Big|_{u = 0, \theta = 0} = 0.
\qedhere \]
\end{proof}

 \subsection{Plancherel measure and bounds on spectral densities}
 
Selberg's trace formula gives a handle on the distribution of the spectrum as the spectral parameters (such as the Laplace eigenvalue) increase, which on quotients of a group such as $G=\PSL_2\mathbb{C}$ is guided by a fixed Plancherel measure, 
depending on $G$ only and supported on the tempered spectrum of $L^2(G)$. In the present situation, keeping in mind the classification from \S\ref{background_sec}, let $\mu_{\pl}$ be the absolutely continuous measure on $i\mathbb{R}\times\mathbb{Z}$ given by
\begin{equation}
\label{plancherel}
\dd\mu_{\pl}(\nu,n)=
\frac1{4\pi^2}(|\nu|^2+n^2)\,|\dd\nu|.
\end{equation}
Then the identity term in Theorem~\ref{EvenTrace} may be rewritten (using Fourier inversion) in the form
\[ \vol(\Gamma\backslash G)\cdot\frac12\int_{i\mathbb{R}}\bigg(\int_{-\infty}^{\infty}g(u)e^{u\nu}\,\dd u\bigg)(\dd\mu_{\pl}(\nu,n)+\dd\mu_{\pl}(\nu,-n)), \]
which should be compared to the cuspidal term (the first sum on the left-hand side of \eqref{tracen}) and may be understood as the leading (or global) term in its geometric expansion, accounting for the spectral density of $G$, and similarly in Theorem~\ref{LL}.

All we need for our application to Theorem~\ref{JointSmallIntervals}
is a uniform estimate on the cardinality of the spectrum in a short window (a familiar step in the derivation of Weyl's law). In this section, we use the even trace formula from Theorem~\ref{EvenTrace} to prove such a local bound for the density of a particular representation type. We emphasize that, while the implied constants in Proposition ~\ref{Weyl} depend on the discrete subgroup $\Gamma$, they are independent of $n$. 
\begin{prop}\label{Weyl}
Let $\Gamma<\PSL_2\mathbb{C}$ be a discrete, co-compact, torsion-free subgroup.
Then, for every $n\in\mathbb{Z}$ and $R\in\mathbb{R}$,
the multiplicities $m_{\Gamma}(\pi_{\nu,n})$ of representations $\pi_{\nu,n}$ in $L^2(\Gamma\backslash G)$
satisfy
\begin{align*}
\sum\limits_{R-1 \leqslant |\nu| \leqslant R+1} m_\Gamma(\pi_{\nu,n})
&\ll\vol(\Gamma\backslash G)\cdot\int_{R-1\leqslant|\nu|\leqslant R+1}\dd\mu_{\pl}(\nu,n)+\mathrm{O}_{\Gamma}(1)\\
&\asymp 
\vol(\Gamma\backslash G)\cdot (R^2+n^2)+\mathrm{O}_{\Gamma}(1).
\end{align*}
\end{prop}

\begin{proof}
Let $g$ be a smooth, even, non-negative, compactly supported function such that $\hat{g}(t)\geqslant 0$ for all $t\in\mathbb{R}\cup i\mathbb{R}$, as well as $\hat g(t) \geqslant 1$ for $|t|\leqslant 1/(2\pi)$. Consider the function
\[ g_R(x) = 2 \cos (R x) g(x). \]
Its Fourier transform $\hat g_R(t) = \hat g(t-R/2\pi) + \hat g( t+R/2\pi)$ has the property that $\hat g_R(t) \geqslant 1$ for $R-1\leqslant 2\pi|t|\leqslant R+1$. Therefore, we have the bound:
\[ \sum\limits_{R-1 \leqslant |\nu| \leqslant R+1} m_\Gamma(\pi_{\nu,n}) \leqslant \sum\limits_{\nu} m_\Gamma(\pi_{\nu,n}) \hat g_R(i\nu/2\pi) + \mathrm{O}_\Gamma(1). \]	

By Theorem~\ref{EvenTrace}, the right-hand side of this estimate equals
\begin{align*}
&-\delta_0(n)\int_{-\infty}^{\infty}g_R(u)e^u\,\dd u+
\frac{1}{2} \delta_{\pm 1}(n) \hat g_R(0)+ \frac{1}{2\pi} \text{vol}(\Gamma \backslash G) (n^2 g_R(0) - g_R''(0))\\
&\qquad +\sum\limits_{[\gamma]} \ell(\gamma_0) |1 - e ^{\mathbb{C} \ell(\gamma)}|^{-1}  |1 - e ^{-\mathbb{C} \ell(\gamma)}|^{-1} g_R(\ell(\gamma)) \cos (n \hol \gamma)+ \mathrm{O}_\Gamma(1).
\end{align*}
The first two terms contribute $\mathrm{O}(1)$. For the fourth term, $g_R$ is a compactly supported function with $\supp g_R\subseteq\supp g$ and $|g_R|\leqslant|g|$, so this sum contains $\mathrm{O}_{\Gamma}(1)$ terms and contributes $\mathrm{O}_{\Gamma}(1)$. Further, we calculate that
\[ g_R''(x) = 2 \cos(R x) g''(x) - 4 R \sin(R x) g'(x) - 2 R^2 \cos(R x) g(x), \]
so that the third term contributes $\frac1{\pi}\vol(\Gamma\backslash G)\cdot((n^2+R^2)g(0)-g''(0))\ll\vol(\Gamma\backslash G)\cdot (R^2+n^2+1)$. Thus,
\[ \sum_{R-1\leqslant|\nu|\leqslant R+1}m_{\Gamma}(\pi_{\nu,n})\ll \vol(\Gamma\backslash G)\cdot (R^2+n^2+1)+\mathrm{O}_{\Gamma}(1), \]
which completes the proof since the remaining claims are immediate.
\end{proof}

An immediate corollary of Proposition~\ref{Weyl} is the following:

\begin{corr}
Let $\Gamma<\PSL_2\mathbb{C}$ be a discrete, co-compact, torsion-free subgroup.
Then, for every $n\in\mathbb{Z}$ and $R\geqslant 1$, the multiplicities $m_{\Gamma}(\pi_{\nu,n})$ of representations $\pi_{\nu,n}$ in $L^2(\Gamma\backslash G)$ satisfy
\[\sum\limits_{|\nu| \leqslant R} m_\Gamma(\pi_{\nu, n}) \ll \vol(\Gamma\backslash G)\cdot (R^3 + n^2 R)+\mathrm{O}_{\Gamma}(R). \]
\end{corr}

\section{Sampling the length spectrum using the trace formula}
\label{trace-formula-sec}

In this section, we prove estimates on certain sums that naturally appear on the geometric side of the trace formulas in Theorems~\ref{EvenTrace} and \ref{OddTrace} when one is interested in sampling geodesics $C_{\gamma}$ on $M$, controlling their length $\ell(\gamma)$ close to or up to a specific length $y$, and detecting their holonomy with a character $e^{in\hol(\gamma)}$.

For suitable even and odd sampling functions $g,h:\mathbb{R}\to\mathbb{R}$, respectively, Theorems~\ref{EvenTrace} and \ref{OddTrace} give a handle on sums
\begin{equation}
\label{smooth-sums}
\begin{aligned}
T_n^{\cos}[g]&=\sum_{[\gamma]}\ell(\gamma_0)g(\ell(\gamma))w(\gamma)\cos(n\hol(\gamma)),\\
T_n^{\sin}[h]&=\sum_{[\gamma]}\ell(\gamma_0)h(\ell(\gamma))w(\gamma)\sin(n\hol(\gamma)),
\end{aligned}
\end{equation}
where the sums are over the non-trivial hyperbolic and loxodromic conjugacy classes $[\gamma]$ of $\Gamma$, $\ell(\gamma_0)$ refers to the length of the geodesic corresponding to the primitive element $\gamma_0$ which generates $\gamma$, and
\begin{equation}
\label{w-gamma-def}
w(\gamma)=|1-e^{\mathbb{C}\ell(\gamma)}|^{-1}|1-e^{-\mathbb{C}\ell(\gamma)}|^{-1}.
\end{equation}

Let $\psi: \mathbb{R}\to\mathbb{R}_{\geqslant 0}$ be a fixed, smooth non-negative function that is compactly supported on $[-1,1]$
and satisfies $\|\psi\|_1=1$ and $\psi\geqslant\frac12$ on $[-\frac12,\frac12]$, and for $\eta>0$, define
\begin{equation}\label{psi-eta-def}
\psi_\eta(t):= \frac{1}{\eta} \psi\Big(\frac{t}{\eta}\Big).
\end{equation}

For $y>0$, define the sampling functions $g_{y,\eta},h_{y,\eta}:\mathbb{R}\to\mathbb{R}$ as convolutions
\begin{equation}
\label{smooth-cutoff-functions}
g_{y,\eta}=\psi_{\eta}\star\chi_{[-y,y]},\quad h_{y,\eta}= \psi_\eta \star (\chi_{[-y,y]} \cdot \sgn).
\end{equation}
In \S\ref{trace-estimates-sec}, we use the non-spherical trace formulas of Theorems~\ref{EvenTrace} and \ref{OddTrace} to prove estimates on $T_n^{\cos}[g_{y,\eta}]$ and $T_n^{\sin}[h_{y,\eta}]$. In \S\ref{sharp-length-sec}, we execute the passage from these smooth counts to the sharp counts
\begin{equation}
\label{sharp-sums}
\begin{alignedat}{5}
&T_n^{\cos}(y)&&:=T_n^{\cos}[\chi_{[-y,y]}]&&=\sum_{\ell(\gamma)\leqslant y}\ell(\gamma_0)w(\gamma)\cos(n\hol(\gamma)),\\
&T_n^{\sin}(y)&&:=T_n^{\sin}[\chi_{[-y,y]}\cdot\sgn] &&= \sum_{\ell(\gamma)\leqslant y}\ell(\gamma_0)w(\gamma)\sin(n\hol(\gamma)).
\end{alignedat}
\end{equation}

\subsection{Trace formula estimates}
\label{trace-estimates-sec}
In this section, we use Theorems~\ref{EvenTrace} and \ref{OddTrace} to prove in Lemma~\ref{trace-estimates-lemma} estimates on the smooth counts $T_n^{\cos}[g_{y,\eta}]$ and $T_n^{\sin}[h_{y,\eta}]$ defined in \eqref{smooth-sums}.

\begin{lemma}
\label{trace-estimates-lemma}
Let $\Gamma<\PSL_2\mathbb{C}$ be a discrete, co-compact, torsion-free subgroup, let $n\in\mathbb{Z}$, and let $y,\eta>0$. Then the sums $T_n^{\cos}[g_{y,\eta}]$ and $T_n^{\sin}[h_{y,\eta}]$ defined in \eqref{smooth-sums}, with $g_{y,\eta}$ and $h_{y,\eta}$ as in \eqref{smooth-cutoff-functions}, satisfy
\begin{align*}
T_n^{\cos}[g_{y,\eta}]&=\delta_0(n)\int_{-\infty}^{\infty}g_{y,\eta}(u)\,\dd\varpi_{\Gamma}^{\ast}(u)+\mathrm{O}_{\Gamma}\left(\frac1{\eta^2}+(1+n^2)\Big(\log^{\ast}\frac1{\eta}+y\Big)\right),\\
T_n^{\sin}[h_{y, \eta}] &= \mathrm{O}_\Gamma\left(\frac{1}{\eta^2} + (1+n^2) \Big(\log^{\ast}\frac{1}{\eta}+y\Big)\right),
\end{align*}
where $\varpi_{\Gamma}^{\ast}$ is as in \eqref{density-star-def} and $\log^{\ast}x=\log(2+x)$.
\end{lemma}

\begin{proof}
Using Theorem~\ref{EvenTrace} for $g_{y,\eta}$, in the form \eqref{EvenTrace-alt}, we know that
\begin{equation}
\label{Tncos-start}
\begin{aligned}
&T_n^{\cos}[g_{y,\eta}]-\delta_0(n)\int_{-\infty}^{\infty}g_{y,\eta}(u)\,\dd\varpi_{\Gamma}^{\ast}(u)
=\frac{1}{2} \sum\limits_{\nu\in i\mathbb{R}} (m_\Gamma(\pi_{\nu,n}) + m_\Gamma(\pi_{\nu,-n})) \hat g_{y,\eta}\Big(\frac{i\nu}{2 \pi}\Big)\\
&\qquad -\frac{1}{2} \delta_{\pm 1}(n) \hat g_{y,\eta}(0) - \frac{1}{2\pi} \text{ vol}(\Gamma \backslash G)(n^2 g_{y,\eta}(0) - g_{y,\eta}''(0)).
\end{aligned}
\end{equation}
With our choice of $g_{y,\eta}$, the right-hand side of \eqref{Tncos-start} equals
\[ \frac{1}{2} \sum\limits_{\nu\in i\mathbb{R}} (m_\Gamma(\pi_{\nu,n}) + m_\Gamma(\pi_{\nu,-n})) \hat g_{y,\eta}\Big(\frac{i\nu}{2 \pi}\Big)-\delta_{\pm 1}(n)y
+\mathrm{O}\Big(\vol(\Gamma\backslash G)\Big(n^2+\frac1{\eta^2}\Big)\Big). \]
(When $\eta< y$, the error term $\mathrm{O}(\vol(\Gamma\backslash G)/\eta^2)$ is not needed here since then $g_{y,\eta}''(0)=0$; however we incur this term later regardless.)

For the spectral contribution, we first compute that the Fourier transform for $\nu=it$ ($t\in\mathbb{R}$) is
\[ \hat g_{y,\eta}\Big(\frac{-t}{2\pi}\Big) = 2 \hat \psi\Big(\frac{-\eta t}{2\pi}\Big) \frac{\sin( ty)}{ t}. \]
Using the Schwartz bound $|\hat \psi(-\eta t/2\pi)|\ll_m1/(1+\eta |t|)^m$ with (say) $m=3$, we obtain
		
	\begin{align*}
	\frac{1}{2} \sum\limits_{\nu=it} m_\Gamma(\pi_{\nu,\pm n})\hat g_{y,\eta}\Big(\frac{i\nu}{2\pi}\Big)
	&=\bigg(\sum_{0\leqslant |t|<1}+\sum_{1\leqslant |t|\leqslant 1/\eta}+\sum_{|t|>1/\eta}\bigg) m_\Gamma(\pi_
	{it,\pm n})
\hat \psi\Big(\frac{-\eta t}{2 \pi}\Big) \frac{\sin(ty)}{ t} \\
	&\ll y\sum_{0\leqslant |t|<1}m_{\Gamma}(\pi_{it,\pm n})+\sum_{1\leqslant k \leqslant 1/\eta} \frac{1}{k} \sum_{k\leqslant |t| < k+1}m_\Gamma(\pi_{it,\pm n})\\
	&\qquad+ \sum_{k >1/{\eta}} \frac{1}{k} \frac{1}{(\eta k)^3} \sum_{k\leqslant |t| < k+1} m_\Gamma(\pi_{it,\pm n}).
	\end{align*}
	
Bounding the multiplicities $m_{\Gamma}(\pi_{\nu,\pm n})$ using the uniform local bound of Proposition ~\ref{Weyl}, we find that the above is 
\begin{align*}
	&\ll\vol(\Gamma\backslash G)\bigg[y(1+n^2)+\sum_{k \leqslant \frac{1}{\eta}} \Big(k + \frac{n^2}{k}\Big) +  \sum_{k > \frac{1}{\eta}} \Big(\frac{1}{\eta^3 k^2} + \frac{n^2}{\eta^3 k^4}\Big)\bigg]\\
	&\qquad+\mathrm{O}_{\Gamma}\bigg(y+\sum_{1\leqslant k\leqslant 1/\eta}\frac1k+\sum_{k>1/\eta}\frac1{\eta^3k^4}\bigg)\\
	&\ll\vol(\Gamma\backslash G)\bigg[y(1+n^2)+\frac{1}{\eta^2} + n^2 \log^{\ast}\frac{1}{\eta}\bigg]+\mathrm{O}_{\Gamma}\Big(y+\log^{\ast}\frac1{\eta}\Big).
	\end{align*}
Putting everything together completes the proof for the even case.

The odd case is completely analogous and in fact easier, since the trace formula from Theorem \ref{OddTrace} has only the principal series spectral and non-identity geometric terms. Indeed, we compute that
\[ \hat{h}_{y,\eta}\Big(\frac{-t}{2\pi}\Big)=4\hat{\psi}\Big(\frac{-\eta t}{2\pi}\Big)\frac{\sin^2(ty/2)}{it}, \]
so that using the Schwartz bound $|\hat \psi(-\eta t/2\pi)|\ll_m1/(1+\eta |t|)^m$ and $|\sin^2(ty/2)/it|\ll\min(ty^2,1/t)\ll\min(y,1/t)$, the estimates proceed as above.
\end{proof}

\begin{remark}
Theorems~\ref{EvenTrace} and \ref{OddTrace} allow for good control over the dependence in $\Gamma$, as we show in Proposition ~\ref{Weyl} and then using this result in the proof of Lemma~\ref{trace-estimates-lemma}. For example, the dependence in the leading terms is often guided by $\vol(\Gamma\backslash G)$ only. This is a very important feature when the group $\Gamma$ varies or where uniformity in $\Gamma$ is required (say, for a varying level in a congruence group). Since $\Gamma$ is fixed for us, from now on we combine all dependence on $\Gamma$, as in the statement of Lemma~\ref{trace-estimates-lemma} and beyond.
\end{remark}

\subsection{Passage to sharp cutoff in length}
\label{sharp-length-sec}
In this section, we pass from the smooth counts for $T_n^{\cos}[g_{y,\eta}]$ and $T_n^{\sin}[h_{y,\eta}]$ of Lemma~\ref{trace-estimates-lemma} to a sharp count for $T_n^{\cos}(y)=T_n^{\cos}[\chi_{[-y,y]}]$ and $T_n^{\sin}(y) = T_n^{\sin}[\chi_{[-y,y]}\cdot\sgn]$ as shown in \eqref{sharp-sums}. This passage requires further use of Theorems~\ref{EvenTrace} and \ref{OddTrace} as well as taking $\eta>0$ small in Lemma~\ref{trace-estimates-lemma}, which typically ends up being the main source of the error terms.

The passage to the sharp geodesic count 
relies primarily on estimating contributions from classes in the transition zone $y-\eta\leqslant\ell(\gamma)\leqslant y+\eta$, which contains the range where $g_{y,\eta}(\ell(\gamma))\neq\chi_{[-y,y]}$. This is achieved in the following key lemma, which is the geometric side analogue, in the length aspect, of Proposition ~\ref{Weyl}.

\begin{lemma}
\label{geodesics-upper-bound}
Let $\Gamma<\PSL_2\mathbb{C}$ be a discrete, co-compact, torsion-free subgroup, and let $y,\eta>0$. Then,
\begin{equation}
\label{geodesics-upper-bound-formula}
\sum_{y-\eta\leqslant\ell(\gamma)\leqslant y+\eta}\ell(\gamma_0)w(\gamma)\ll \int_{y-2\eta}^{y+2\eta}\dd\varpi_{\Gamma}^{\ast}(u)+\mathrm{O}_{\Gamma}\Big(\frac1{\eta^2}+\eta\Big),
\end{equation}
where $w(\gamma)$ is as in \eqref{w-gamma-def} and $\varpi_{\Gamma}^{\ast}$ is as in \eqref{density-star-def}.
\end{lemma}

\begin{proof}
We will sample the geodesics in the range $\ell(\gamma)\in [y-\eta,y+\eta]$ using an even, majorant function
\begin{equation}
\label{majorant-f}
f_{y,2\eta}(x)=\psi\Big(\frac{x-y}{2\eta}\Big)+\psi\Big(\frac{x+y}{2\eta}\Big),
\end{equation}
with a smooth, non-negative bump function $\psi$ as in \eqref{psi-eta-def}. By the definition of $\psi$, $f_{y,2\eta}:\mathbb{R}\to\mathbb{R}$ is a smooth, even, non-negative, absolutely bounded function supported on $\pm[y-2\eta,y+2\eta]$ and satisfying $f_{y,2\eta}\gg 1$ on $\pm [y-\eta,y+\eta]$, so that
\[ \sum_{y-\eta\leqslant\ell(\gamma)\leqslant y+\eta}\ell(\gamma_0)w(\gamma)\ll\sum_{[\gamma]}\ell(\gamma_0)w(\gamma)f_{y,2\eta}(\ell(\gamma)). \]
	
Using the even trace formula of Theorem~\ref{EvenTrace}, with $n=0$ and in the form \eqref{EvenTrace-alt}, we obtain
\begin{align*}
&\sum_{[\gamma]}\ell(\gamma_0)w(\gamma)f_{y,2\eta}(\ell(\gamma))\\
&\qquad =
\int_{-\infty}^\infty f_{y,2\eta}(u)\,\dd\varpi_{\Gamma}^{\ast}(u) + \sum_{\nu\in i\mathbb{R}}m_\Gamma(\pi_{\nu,0}) \hat f_{y, 2 \eta}\Big(\frac{i \nu}{2\pi}\Big) + \frac{1}{2 \pi} \text{ vol}(\Gamma \backslash G) f_{y, 2 \eta}''(0).
\end{align*}
Note that, at this point, we are only using the familiar spherical trace formula. We estimate the first term using $f_{y,2\eta}\ll 1$ and the support condition on $f_{y,2\eta}$. The third term may be absorbed in $\mathrm{O}_{\Gamma}(1/\eta^2)$. (In fact, in a typical application with $\eta<y/2$, this term vanishes.)

For the principal series representations $\pi_{it, 0}$ ($t \in \mathbb{R}$), the Fourier transform $\hat f_{y, 2 \eta}(t)= 4\eta \cos(2 \pi ty) \hat \psi(2\eta t)$ satisfies the Schwartz bound $|\hat f_{y, 2 \eta}(-t/2\pi)|\ll \eta/(1+\eta|t|)^4$. Bounding the multiplicities $m_{\Gamma}(\pi_{it,0})$ using the uniform bound of Proposition ~\ref{Weyl}, we may finally bound the contribution of the principal series representations as
\begin{align*}
\sum_{\nu=it}\hat f_{y, 2 \eta}(-t/2\pi) m_\Gamma(\pi_{it,0}) &=
\bigg(\sum_{0\leqslant k<1/\eta}+\sum_{k>1/\eta}\bigg)\sum\limits_{k\leqslant |t| < k+1} \hat f_{y, 2 \eta}(-t/2\pi) m_\Gamma(\pi_{it,0})\\
&\ll_{\Gamma}\eta \sum_{0\leqslant k< 1/\eta} (k^2+1) + \frac{1}{\eta^3} \sum_{k>1/{\eta}} \frac{1}{k^2}\ll\frac1{\eta^2}+\eta.
\end{align*}
Combining everything completes the proof.
\end{proof}

\begin{remark}
\label{typical-regime}
In the typical regime for the application of Lemma~\ref{geodesics-upper-bound}, when $\eta\ll 1\ll y-\eta$, we have in \eqref{geodesics-upper-bound-formula} simply $w(\gamma)\asymp_\Gamma e^{-\ell(\gamma)}$ and $|\dd\varpi_{\Gamma}^{\ast}(u)/\dd u|\asymp_{\Gamma} e^u$, so that Lemma~\ref{geodesics-upper-bound} states that
\[ \sum_{y-\eta\leqslant\ell(\gamma)\leqslant y+\eta}\ell(\gamma_0)\ll_{\Gamma}e^y\Big(\eta e^y+\frac1{\eta^2}\Big). \]
This will be the case, in particular, in the proof of Proposition ~\ref{trace-estimates-sharp}.
\end{remark}

\begin{prop}
\label{trace-estimates-sharp}
Let $\Gamma<\PSL_2\mathbb{C}$ be a discrete, co-compact, torsion-free subgroup, and let $n\in\mathbb{Z}$. Then, for every $y>0$, the sums $T_n^{\cos}(y)$ and $T_n^{\sin}(y)$ defined in \eqref{sharp-sums} satisfy
\begin{align*}
T_n^{\cos}(y)&=\delta_0(n)\int_{-y}^y\dd\varpi_{\Gamma}^{\ast}(u)+\mathrm{O}_{\Gamma}\big(e^{2y/3}+n^2y\big),\\
T_n^{\sin}(y)&=\mathrm{O}_{\Gamma}\big(e^{2y/3}+n^2y\big).
\end{align*}
\end{prop}

\begin{proof}
We will use Lemmata~\ref{trace-estimates-lemma} and \ref{geodesics-upper-bound}, with a parameter $\eta>0$ to be suitably chosen momentarily. According to the definition \eqref{smooth-cutoff-functions}, we have that the functions $g_{y,\eta}$ and $\chi_{[-y,y]}$ agree outside the set $\pm [y-\eta,y+\eta]$, on which $|g_{y,\eta}-\chi_{[-y,y]}|=\mathrm{O}(1)$. Therefore,
\begin{align*}
\big|T_n^{\cos}(y)-T_n^{\cos}[g_{y,\eta}]\big|
&=\bigg|\sum_{[\gamma]}\ell(\gamma_0)\big(g_{y,\eta}(\ell(\gamma))-\chi_{[-y,y]}(\ell(\gamma))\big)w(\gamma)\cos(n\hol(\gamma))\bigg|\\
&\ll\sum_{y-\eta\leqslant\ell(\gamma)\leqslant y+\eta}\ell(\gamma_0)w(\gamma).
\end{align*}

Therefore, using Lemmata~\ref{trace-estimates-lemma} and \ref{geodesics-upper-bound}, we find that
\begin{align*}
&\bigg|T_n^{\cos}(y)-\delta_0(n)\int_{-y}^y\dd\varpi_{\Gamma}^{\ast}(u)\bigg|\\
&\qquad\leqslant \big|T_n^{\cos}(y)-T_n^{\cos}[g_{y,\eta}]\big|+\bigg|T_n^{\cos}[g_{y,\eta}]-\delta_0(n)\int_{-\infty}^{\infty}g_{y,\eta}(u)\,\dd\varpi_{\Gamma}^{\ast}(u)\bigg|\\
&\qquad\qquad+\delta_0(n)\int_{-\infty}^{\infty}\big(g_{y,\eta}(u)-\chi_{[-y,y]}(u)\big)\,\dd\varpi_{\Gamma}^{\ast}(u)\\
&\qquad\ll\int_{y-2\eta}^{y+2\eta}\dd\varpi_{\Gamma}^{\ast}(u)+\mathrm{O}_{\Gamma}\bigg(\frac1{\eta^2}+(1+n^2)\Big(\log^{\ast}\frac1{\eta}+y\Big)+\eta\bigg).
\end{align*}

The statement of Proposition ~\ref{trace-estimates-sharp} is vacuously true for $y=\mathrm{O}_{\Gamma}(1)$, so we may assume that $y\gg_{\Gamma}1$. As already mentioned in Remark~\ref{typical-regime} and is clear from the definition \eqref{density-star-def}, $\dd\varpi_{\Gamma}^{\ast}(u)/\dd u\asymp_{\Gamma} e^u$ for $u\geqslant 0$ (or $u=\mathrm{O}_{\Gamma}(1)$). We will choose $\eta\ll_{\Gamma}1$; then,
\[ \int_{y-2\eta}^{y+2\eta}\dd\varpi_{\Gamma}^{\ast}(u)\asymp_{\Gamma} \eta e^y. \]
The admissible choice $\eta=e^{-y/3}$ optimizes the error terms and yields Proposition ~\ref{trace-estimates-sharp} for $T_n^{\cos}$.

Similarly, if we additionally require $2\eta$ to be less than the minimal geodesic length $\eta_0(\Gamma)$, we can show that 
\[\big|T_n^{\sin}(y) - T_n^{\sin}[h_{y,\eta}]\big| \ll \sum_{y-\eta\leqslant\ell(\gamma)\leqslant y+\eta}\ell(\gamma_0)w(\gamma), \]
and the proof is identical from here on.
\end{proof}

\subsection{Passage to sharp cutoff in holonomy}
\label{sharp-holonomy-sec}
In this section, we prepare the ground for passage to the sharp count in holonomy. The path to the sharp count is again the geometric counterpart to Proposition ~\ref{Weyl}, but this time in the holonomy aspect.

\begin{lemma}
\label{passage-to-sharp-holonomy}
Let $\Gamma<\PSL_2\mathbb{C}$ be a discrete, co-compact, torsion-free subgroup, and let $y>0$, $\theta_0\in\mathbb{R}$, and $0<\eta' \leqslant 2\pi$. Then,
\begin{equation}
\label{holonomy-upper-bound-formula}
\sum_{\substack{\ell(\gamma)\leqslant y\\ \theta_0-\eta'\leqslant\hol(\gamma)\leqslant\theta_0+\eta'}}\ell(\gamma_0)w(\gamma)\ll \eta' \int_{-y}^y\dd\varpi_{\Gamma}^{\ast}(u)+\mathrm{O}_{\Gamma}\Big(\frac y{\eta'^2}\Big),
\end{equation}
where $w(\gamma)$ is as in \eqref{w-gamma-def} and $\varpi_{\Gamma}^{\ast}$ is as in \eqref{density-star-def}.
\end{lemma}

\begin{proof}
Since $\Gamma$ is discrete, the claim is vacuously true for $y<\eta_0(\Gamma)$, so we may assume that $y\geqslant\eta_0(\Gamma)$. Consider a majorant function $f_{\theta_0,2\eta'}:\mathbb{R}/2\pi\mathbb{Z}\to\mathbb{R}$ given by
\[ f_{\theta_0,2\eta'}(t)=\sum_{n\in\mathbb{Z}}\Big[\psi\Big(\frac{t+2n\pi-\theta_0}{2\eta'}\Big)+\psi\Big(\frac{t+2n\pi+\theta_0}{2\eta'}\Big)\Big], \]
where $\psi$ is a smooth, non-negative bump function as in \eqref{psi-eta-def}. This is simply a $2\pi\mathbb{Z}$-periodization of the majorant \eqref{majorant-f} used in the proof of Lemma~\ref{geodesics-upper-bound}; it is a smooth, even, non-negative absolutely bounded function supported on $\pm [\theta_0-2\eta',\theta_0+2\eta']+2\pi\mathbb{Z}$ and satisfying $f_{\theta_0,2\eta'}\gg 1$ on $\pm [\theta_0-\eta',\theta_0+\eta']+2\pi\mathbb{Z}$. Since the non-negative sampling function $g_{y+1,1}=\psi_1\star\chi_{[-y-1,y+1]}$ given in \eqref{smooth-cutoff-functions} also satisfies $g_{y+1,1}\gg 1$ for $[-y,y]$, we have
\[ \sum_{\substack{\ell(\gamma)\leqslant y\\ \theta_0-\eta'\leqslant\hol(\gamma)\leqslant\theta_0+\eta'}}\ell(\gamma_0)w(\gamma)\ll\sum_{[\gamma]}\ell(\gamma_0)w(\gamma)g_{y+1,1}(\ell(\gamma))f_{\theta_0,2\eta'}(\hol(\gamma)). \]

Using the cosine Fourier expansion for $f_{\theta_0,2\eta'}$ shows that, after a standard unfolding calculation, the right-hand side equals
\begin{align*}
&\sum_{[\gamma]}\ell(\gamma_0)w(\gamma)g_{y+1,1}(\ell(\gamma))\cdot\frac{4\eta'}{2\pi}\sum_{m\in\mathbb{Z}}\cos(m\theta_0)\hat{\psi}\big(\frac{\eta' m}{\pi}\big)\cos(m\hol(\gamma))\\
&\qquad=\frac{2\eta'}{\pi}\sum_{m\in\mathbb{Z}}\cos(m\theta_0)\hat{\psi}\big(\frac{\eta' m}{\pi}\big)T_m^{\cos}[g_{y+1,1}],
\end{align*}
using the definition \eqref{smooth-sums}. Applying Lemma~\ref{trace-estimates-lemma} and the Schwartz estimate $|\hat{\psi}(\eta' t/\pi)|\ll_k 1/(1+\eta'|t|)^k$ with (say) $k=4$, this sum evaluates as
\begin{align*}
&\frac{2\eta'}{\pi}\hat{\psi}(0)\int_{-\infty}^{\infty}g_{y+1,1}(u)\,\dd\varpi_{\Gamma}^{\ast}(u)
+\mathrm{O}_{\Gamma}\bigg(\eta'\sum_{m\in\mathbb{Z}}\Big|\hat{\psi}\Big(\frac{\eta' m}{
\pi}\Big)\Big|(1+m^2)y\bigg)\\
&\qquad\ll\eta'\int_{-y-2}^{y+2}\,\dd\varpi_{\Gamma}^{\ast}(u)+\mathrm{O}_{\Gamma}\bigg(
\eta' y\bigg[1+\sum_{1\leqslant m\leqslant 1/\eta'}m^2+\sum_{m>1/\eta'}\frac{m^2}{(\eta' m)^4}\bigg]\bigg)\\
&\qquad\ll\eta'\int_{-y}^y\,\dd\varpi_{\Gamma}^{\ast}(u)+\mathrm{O}_{\Gamma}\Big(\frac{y}{\eta'^2}\Big),
\end{align*}
which completes the proof.
\end{proof}

\begin{remark}
\label{symmetry-remark}
The reader will notice that this device for passage to sharp count in holonomy requires the use of Lemma~\ref{trace-estimates-lemma} (and thus Theorem~\ref{EvenTrace}) with large $n$. This mirrors the fact that the passage to sharp count in length requires the use of Lemma~\ref{trace-estimates-lemma} with small $\eta>0$, which in turn relies on using Theorem~\ref{EvenTrace} with large spectral parameter $\nu$. In both cases, the proof boils down to estimates on the density of the automorphic spectrum $\pi_{\nu,n}$ with spectral parameters increasing in different directions; these are in turn provided by Proposition ~\ref{Weyl} which works over any ball of spectral parameters of radius $\mathrm{O}(1)$. This structural parallel underlies the agreement between our results in the length and holonomy aspects.
\end{remark}

\section{Primitivity and weights}\label{PreliminariesSection}

For $y>0$ and $n\in\mathbb{Z}$, consider the sums
\begin{equation}
\label{SnSnP}
S_n(y)= \sum_{\ell(\gamma) \leqslant y} \ell(\gamma) e^{-\ell(\gamma)+ i n \hol \gamma}, \qquad
S_n^{P}(y)= \sump_{\ell(\gamma) \leqslant y} \ell(\gamma) e^{-\ell(\gamma)+i n \hol \gamma},
\end{equation}
where the summation is over the non-trivial conjugacy classes $[\gamma]$ of $\Gamma$, and (here and throughout) the superscript $^P$ indicates that summation is restricted to primitive non-trivial conjugacy classes. The sum $S_n^P(y)$ is of primary interest for counting primitive geodesics with control on holonomy.

On the other hand, an application of trace formula as in Section~\ref{trace-formula-sec} (see \eqref{sharp-sums} and Proposition ~\ref{trace-estimates-sharp}) naturally gives a handle on sums such as $T_n(y)$ and its cousin $T_n^P(y)$ defined by
\begin{equation}
\label{TnTnP}
T_n(y) = \sum_{\ell(\gamma) \leqslant y} \ell(\gamma_0) w(\gamma) e^{i n \hol(\gamma)}, \qquad 
T_n^P(y) = \sump_{\ell(\gamma) \leqslant y} \ell(\gamma) w(\gamma) e^{i n \hol(\gamma)},
\end{equation}
where $\gamma_0$ is the primitive hyperbolic or loxodromic element that generates $\gamma$ and the weight $w(\gamma)\asymp_{\Gamma}e^{-\ell(\gamma)}$ is as in \eqref{w-gamma-def}. The main result of this section, Lemma~\ref{weights-primitivity} shows that all four sums defined in \eqref{SnSnP} and \eqref{TnTnP} agree up to a very small error term. In Lemma~\ref{weights-primitivity-smooth}, we record a similar result for sums with a more general length cutoff.

\begin{lemma}
\label{weights-primitivity}
Let $\Gamma<\PSL_2\mathbb{C}$ be a discrete, co-compact, torsion-free subgroup, and let $y>0$, $n\in\mathbb{Z}$. Then the sums defined in \eqref{SnSnP} and \eqref{TnTnP} satisfy:
\begin{alignat}{3}
\label{G=GP} S_n^P(y) &= S_n(y) &&+ \mathrm{O}_\Gamma(y),\\
\label{HP=GP} S_n^P(y) &= T_n^P(y) &&+ \mathrm{O}_\Gamma(y),\\
\label{Hn=HP} T_n^P(y) &= T_n(y) &&+ \mathrm{O}_\Gamma(y),
\end{alignat}
and, consequently,
\begin{equation}
\label{Hn=Gn} S_n^P(y)=T_n(y)+\mathrm{O}_{\Gamma}(y).
\end{equation}
\end{lemma}

\begin{remark}
The bound \eqref{Hn=Gn}, which follows from \eqref{HP=GP} and \eqref{Hn=HP}, can be thought of as a statement about removal of unwieldy weights $w(\gamma)$ and imprimitive classes from $T_n(y)$. Since our bounds on $T_n(y)$ in Proposition ~\ref{trace-estimates-sharp} are exponential in $y$ (in particular $T_0(y)\asymp_{\Gamma}e^y$), Lemma~\ref{weights-primitivity} shows that the error terms introduced by these maneuvers are very small in comparison, and also that (cf.~\eqref{G=GP}) all our statements hold if extended to include imprimitive  geodesics. The sources of leading error terms in all our principal results are elsewhere, notably in the passage from smooth to sharp cutoff.

Lemma~\ref{weights-primitivity} is similar, both in spirit and quantitative strength, to the relations between 
Chebyshev's functions in the proof of the Prime Number Theorem. The estimate $S_n^P(y)=S_n(y)+\mathrm{O}_{\Gamma}(y)$ should be compared to the classical estimate $|\psi(x)-\theta(x)|\ll x^{1/2}$ for $\psi(x),\theta(x)\sim x$ as $x\to\infty$, bearing in mind the weights $\ell(\gamma)e^{-\ell(\gamma)}$ in \eqref{SnSnP}. If those weights were removed by summation by parts as in the proof of Proposition~\ref{holonomy-sums-prop}, the corresponding unweighted sums (say $\tilde{S}_n(y)$ and $\tilde{S}_n^P(y)=K_n(y)$ in \eqref{holonomy-character-sums}) would satisfy $|\tilde{S}_n(y)-\tilde{S}_n^P(y)|\ll_{\Gamma,n}e^y$ as compared to the main term $\tilde{S}_0^P(y)\sim_{\Gamma} e^{2y}/2y$.

Essential ideas for Lemma~\ref{weights-primitivity} are due to Sarnak--Wakayama~\cite[Lemmata~7.1, 7.2]{SarWa}. As for us, a key element of their proof of the equidistribution result~\eqref{SarnakWakayamaResult} is to approximate $S_n^P(y)$, which appears in the spectral (Fourier) decomposition of the sum over primitive conjugacy classes, by $T_n(y)$, which can be approximated by the hyperbolic/loxodromic term in the non-spherical trace formula with an appropriate choice of test function. For a general rank one locally symmetric space of finite volume and negative curvature, Sarnak and Wakayama show that all four sums \eqref{SnSnP}--\eqref{TnTnP} have the same asymptotic growth, with a sub-exponential error term. Lemma~\ref{weights-primitivity} explicates and sharpens this error term in the context of compact hyperbolic 3-manifolds.
\end{remark}

\begin{proof}
Note that since $\Gamma$ is discrete, there is a minimum geodesic length $\eta_0(\Gamma)$. When $y <\eta_0(\Gamma)$, all of the sums \eqref{SnSnP}--\eqref{TnTnP} vanish, and Lemma~\ref{weights-primitivity} holds vacuously; thus we may assume that $y\geqslant\eta_0(\Gamma)$.

We will require, for $k\geqslant 2$ and $y>0$, an estimate on the sum $J_k(y)$ defined by
\[ J_k(y) := \sump_{\ell(\gamma) \leqslant y} \ell(\gamma) e^{-k \ell(\gamma)}. \] 
Since $w(\gamma)\asymp_{\Gamma} e^{-\ell(\gamma)}$, Proposition ~\ref{trace-estimates-sharp} with $n=0$ shows that
\[ J^{\ast}(y):=\sump_{\ell(\gamma) \leqslant y} \ell(\gamma) e^{-\ell(\gamma)}\ll_{\Gamma} e^y. \]
In fact, a more precise bound $J^{\ast}(y)\ll e^y+\mathrm{O}_{\Gamma}\big(e^{\frac23y}+e^{\nu_1y}\big)$, with $\nu_1$ as in \eqref{PGT} and to be omitted if $\Gamma$ admits no complementary spectrum, follows from Proposition ~\ref{trace-estimates-sharp} with $n=0$ or (essentially) from the Prime Geodesic Theorem~\eqref{PGT}, but we will not need this. Using integration by parts,
\begin{equation}
\label{Jky}
\begin{aligned}
J_k(y)&=\int_{\eta_0(\Gamma)^{-}}^{y+}e^{-(k-1)t}\,\dd J^{\ast}(t)
=e^{-(k-1)t}J^{\ast}(t)\bigg|_{\eta_0(\Gamma)^{-}}^{y+}\!\!+(k-1)\int_{\eta_0(\Gamma)}^yJ^{\ast}(t)e^{-(k-1)t}\,\dd t\\
&\ll_{\Gamma}\delta_{2}(k)y+e^{-k\eta_0(\Gamma)},
\end{aligned}
\end{equation}
where $\delta_2$ is as in \S\ref{notation-subsec}.
With the estimate \eqref{Jky} at our disposal, we proceed to prove \eqref{G=GP}--\eqref{Hn=HP}.

\textbf{Proof of \eqref{G=GP}:} Suppose $\tilde \gamma \in \Gamma$ is not primitive; then $\tilde \gamma = \gamma^k$ for some primitive $\gamma \in \Gamma$ and $k\geqslant 2$. The length and holonomy are $\ell(\gamma^k) = k \ell(\gamma)$ and $\hol(\gamma^k) = k \hol(\gamma)$, respectively. Therefore,
	\begin{align*}
	|S_n(y) - S_n^P(y)| &= \Big|\sum_{k \geqslant 2}\,\,\, \sump_{\ell(\gamma) \leqslant y/k}  k \ell(\gamma) e^{-k\ell(\gamma) + i kn \hol(\gamma)}\Big|\\
	&\leqslant \sum_{k \geqslant 2} k \sump_{\ell(\gamma) \leqslant y/k} \ell(\gamma) e^{-k \ell(\gamma)}
	=\sum_{k \geqslant 2} k J_k(y/k),
	\end{align*}
where in fact the sum truncates at $k\leqslant y/\eta_0(\Gamma)$. Using the estimate \eqref{Jky}, we find that
\[ |S_n(y) - S_n^P(y)|\leqslant\sum_{k \geqslant 2} k J_k(y/k)  = \mathrm{O}_{\Gamma}(y), \]
as required.

\textbf{Proof of \eqref{HP=GP}:} First, we have that
\[ T_n^P(y) - S_n^P(y) = \sump_{\ell(\gamma) \leqslant y} \ell(\gamma) (w(\gamma)- e^{-\ell(\gamma)}) e^{i n \hol(\gamma)}. \]
To simplify, observe that 
\begin{align*}
w(\gamma)&=|1- e^{\mathbb{C}\ell(\gamma)}|^{-1}|1- e^{-\mathbb{C}\ell(\gamma)}|^{-1} = e^{- \ell(\gamma)} |1 - e^{-\mathbb{C} \ell(\gamma)}|^{-2}\\
&=e^{-\ell(\gamma)}\big(1+\mathrm{O}_{\Gamma}(e^{-\ell(\gamma)})\big).
\end{align*}
Therefore, using again the estimate \eqref{Jky},
\[ |T_n^P(y) - S_n^P(y)| \ll_{\Gamma} \sump_{\ell(\gamma) \leqslant y} \ell(\gamma) e^{-2\ell(\gamma)}=J_2(y)\ll_{\Gamma} y. \]

\textbf{Proof of \eqref{Hn=HP}:} As in the proof of \eqref{G=GP}, we have
\[ T_n(y) - T_n^P(y)=\sum_{k\geqslant 2}\,\,\,\sump_{\ell(\gamma) \leqslant y/k} 
\ell(\gamma) w(\gamma^k) e^{ikn\hol(\gamma)}. \]
Since $w(\gamma^k)=e^{-k\ell(\gamma)}(1+\mathrm{O}_{\Gamma}(e^{-k\ell(\gamma)}))$, we thus have that
\[ T_n(y) - T_n^P(y)\ll_{\Gamma}\sum_{k \geqslant 2} 
\,\,\,
\sump_{\ell(\gamma) \leqslant y/k} \ell(\gamma) e^{-k \ell(\gamma)}
	=\sum_{k \geqslant 2} 
	J_k(y/k)=\mathrm{O}_{\Gamma}(y), \]
as in the proof of \eqref{G=GP}.
\end{proof}

For future reference, we also include a version of Lemma~\ref{weights-primitivity} with more arbitrary (such as smooth) cutoffs. This presents no serious distinction, as the proof of Lemma~\ref{weights-primitivity} uses the cutoff only to control the set of geodesics entering the estimates, followed by term-wise estimates. For $n\in\mathbb{Z}$ and a bounded, compactly supported function $g:\mathbb{R}\to\mathbb{R}$, define
\begin{equation}
\label{SnTn-smooth}
\begin{aligned}
S_n^P[g]&=\sump_{[\gamma]}\ell(\gamma)g(\ell(\gamma))e^{-\ell(\gamma)+i n \hol \gamma},\\
T_n[g]&=\sum_{[\gamma]}\ell(\gamma_0)g(\ell(\gamma))w(\gamma)e^{i n \hol \gamma},
\end{aligned}
\end{equation}
with notation as in \eqref{SnSnP} and \eqref{TnTnP}. The following lemma shows that these two sums are also comparatively very close.

\begin{lemma}
\label{weights-primitivity-smooth}
Let $\Gamma<\PSL_2\mathbb{C}$ be a discrete, co-compact, torsion-free subgroup, and let $g:\mathbb{R}\to\mathbb{R}$ be a bounded function supported in $[-y,y]$. Then the sums $S_n^P[g]$ and $T_n[g]$ defined in \eqref{SnTn-smooth} satisfy
\[ S_n^P[g]=T_n[g]+\mathrm{O}_{\Gamma}(\|g\|_{\infty}y). \]
\end{lemma}

\begin{proof}
Defining $S_n[g]$ and $T_n^P[g]$ in the obvious way, we find that
\begin{align*}
\big|T_n[g]-T_n^P[g]\big|&=\bigg|\Big(\sum-\sump\Big)_{[\gamma]}\ell(\gamma_0)g(\ell(\gamma))w(\gamma)e^{in\hol(\gamma)}\bigg|\\
&\leqslant\|g\|_{\infty}\Big(\sum-\sump\Big)_{\ell(\gamma)\leqslant y}\ell(\gamma_0)w(\gamma).
\end{align*}
From here, an identical proof to the proof of \eqref{Hn=HP} in Lemma~\ref{weights-primitivity} shows that
\[ T_n[g]-T_n^P[g]\ll_{\Gamma} \|g\|_{\infty}y. \]
The proofs that all four sums $S_n[g]$, $S_n^P[g]$, $T_n^P[g]$ and $T_n[g]$ are within $\mathrm{O}_{\Gamma}(\|g\|_{\infty}y)$ of each other follow in the same way by bootstrapping the proof of Lemma~\ref{weights-primitivity}, \emph{mutatis mutandis}.
\end{proof}

\section{Holonomy character sums}\label{character sums}

In this section, we prove estimates on the ``holonomy character sums''
\begin{equation}
\label{holonomy-character-sums}
K_n[g_{y,\eta}]=\sump_{[\gamma]}g_{y,\eta}(\ell(\gamma))e^{in\hol(\gamma)},\qquad K_n(y)=\sump_{\ell(\gamma)\leqslant y} e^{in\hol(\gamma)},
\end{equation}
where $y,\eta>0$, the cutoff function $g_{y,\eta}$ is as in \eqref{smooth-cutoff-functions}, and the summation is over all non-trivial primitive hyperbolic and loxodromic conjugacy classes $[\gamma]$ of $\Gamma$.

The sums $K_n[g_{y,\eta}]$ and $K_n(y)$ capture the primitive length spectrum of $\Gamma$ with a smooth and sharp cutoff up to around $y>0$, respectively, weighted by characters $\chi_{0,n}$ (see \eqref{characters-T}) on the holonomy group $T\cap \PSU_2
\simeq\mathbb{R}/2\pi\mathbb{Z}$. They play an analogous role to that of Dirichlet character sums in the context of the Prime Number Theorem in Arithmetic Progressions; in particular, asymptotics for the sum $K_0(y)$ recover the Prime Geodesic Theorem, while for $n\neq 0$ the sums $K_n(y)$ feature substantial cancellation, as shown in the following proposition.

\begin{prop}[Holonomy character sums]
\label{holonomy-sums-prop}
Let $\Gamma<\PSL_2\mathbb{C}$ be a discrete, co-compact, torsion-free subgroup, and let $y>0$, $n\in\mathbb{Z}$, and $0<\eta\leqslant\eta_0$. Then the sums $K_n[g_{y,\eta}]$ and $K_n(y)$ defined in \eqref{holonomy-character-sums} satisfy
\begin{equation}
\label{hol-char-sum-result}
\begin{aligned}
K_n[g_{y,\eta}] &=\delta_0(n)\int_2^{\infty}g_{y,\eta}(u)\,\dd\varpi_{\Gamma}(u)+\mathrm{O}_{\Gamma,\eta_0}\bigg(e^{ y} \Big(\frac1{y\eta^2}+n^2+1 \Big)\bigg),\\
K_n(y) &= \delta_0(n)\int_2^y\dd\varpi_{\Gamma}(u)+\mathrm{O}_{\Gamma}\Big(\frac{e^{5y/3}}y + n^2 e^y\Big),
\end{aligned}
\end{equation}
where $\delta_0$ and $\varpi_{\Gamma}$ are as in \S\ref{notation-subsec} and \eqref{ei-varpi}.
\end{prop}

\begin{proof}
We begin with $K_n(y)$, which is technically simpler. Recall the sum $S_n^P(y)$ defined in \eqref{SnSnP}. Combining Proposition ~\ref{trace-estimates-sharp} and Lemma~\ref{weights-primitivity}, we may write
\begin{gather}
S_n^P(y)=\sump_{\ell(\gamma)\leqslant y}\ell(\gamma)e^{-\ell(\gamma)+in\hol(\gamma)}
=\delta_0(n)\int_{-y}^y\dd\varpi_{\Gamma}^{\ast}(u)+s_n^P(y),\nonumber\\
s_n^P(y)=\mathrm{O}_{\Gamma}\big(e^{2y/3}+n^2y\big).\label{snP-estimate}
\end{gather}
Since $\Gamma$ is discrete, there is a minimum geodesic length $\eta_0(\Gamma)$. With an eye toward summation by parts, we first rewrite $K_n(y)$ as an integral, separating out the principal part:
\begin{equation}
\label{Kny-rewritten}
K_n(y)=\int_{\eta_0(\Gamma)-}^{y+}\frac{e^t}t\,\dd S_n^P(t)=\delta_0(n)\int_{\eta_0(\Gamma)-}^{y+}\frac{e^t}t\frac{\dd}{\dd t}\int_{-t}^t\dd\varpi_{\Gamma}^{\ast}(u)\,\dd t+\int_{\eta_0(\Gamma)-}^{y+}\frac{e^t}t\,\dd s_n^P(t).
\end{equation}
Recalling \eqref{ei-varpi} and \eqref{density-star-def}, the first term equals
\begin{equation}
\label{Kny-part1}
\delta_0(n)\!\int_{\eta_0(\Gamma)
}^{y}\frac{e^t}t\Big(\frac{\dd \varpi_\Gamma^{\ast}}{\dd u}\Big|_{u = t} + \frac{\dd \varpi_\Gamma^{\ast}}{\dd u}\Big|_{u = -t} \Big) \, \dd t
=\delta_0(n)\!\int_2^y\dd\varpi_{\Gamma}(u)+ 
\mathrm{O}_{\Gamma}\Big(\log^{\ast}y+\frac{e^{(1-\nu_k)y}}y\Big),
\end{equation}
recalling from \eqref{density-star-def} the definition of $\dd \varpi_\Gamma^\ast$ and from \eqref{ei-varpi} the notation $1-\nu_j^2$ for the exceptional eigenvalues of the Laplacian, with $0 < \nu_k \leqslant \dots \leqslant \nu_1 <1$ (and we set formally $\nu_0=1$ if $k=0$). Using integration by parts and the estimate \eqref{snP-estimate}, the second term in \eqref{Kny-rewritten} is
\begin{equation}
\label{Kny-part2}
\begin{aligned}
\frac{e^t}{t} s_n^P(t)\bigg|_{\eta_0(\Gamma)-}^{y+} \!\! - \int_{\eta_0(\Gamma)}^y \Big(\frac{e^t}{t}- \frac{e^t}{t^2}\Big) s_n^P(t)\,\dd t
&\ll_{\Gamma} \frac{e^{5y/3}}y+n^2e^y+\int_{\eta_0(\Gamma)}^y\Big(\frac{e^{5t/3}}t+n^2e^t\Big)\,\dd t\\
&\ll_{\Gamma} \frac{e^{5y/3}}y+n^2e^y.
\end{aligned}
\end{equation}
Combining \eqref{Kny-part1} and \eqref{Kny-part2} gives the desired asymptotic for $K_n(y)$.

Now, we turn to $K_n[g_{y,\eta}]$. This case presents a minor technical difficulty in that the first step in summation by parts \eqref{Kny-rewritten}
does not work as cleanly. To address this, we rework the proof of Lemma \ref{trace-estimates-lemma} 
by adjusting the choice of the test function $g_{y,\eta}$ from \eqref{smooth-cutoff-functions}
to a slightly different test function $g^{\lambda}_{t,\eta}:\mathbb{R}\to\mathbb{C}$ given, for $0<\eta\leqslant \eta_0\leqslant t$ and $\lambda\in [0,1]$ as follows:
\begin{equation}
\label{def-modified-g}
\psi_{\eta}^{\lambda}(x)=\psi_{\eta}(x)e^{\lambda x},\quad g_{t,\eta}^{\lambda}=\chi_{[-\eta,t]}\star\psi_{\eta}^{\lambda}\,+\,\chi_{[-t,\eta]}\star\psi_{\eta}^{-\lambda},
\end{equation}
where $\psi_\eta$ is as defined in $\eqref{psi-eta-def}$. We remark that the asymptotic for $K_n[g_{y,\eta}]$ in \eqref{hol-char-sum-result} holds trivially for $y=\mathrm{O}(1)$ (due to the discreteness of $\Gamma$), so from now on we may assume that $y\geqslant\eta_0$.
We compute, using the Schwartz bound for $\hat{\psi}$,
\begin{align*}
\hat{g}_{t,\eta}^{\lambda}(\xi)=\sum\nolimits_{\pm}\hat{\chi}_{\pm[-\eta,t]}(\xi)\hat{\psi}^{\pm\lambda}_{\eta}(\xi)
&=\sum\nolimits_{\pm}\frac{e^{\pm 2\pi i\eta\xi}-e^{\mp 2\pi it\xi}}{\pm 2\pi i\xi}\hat{\psi}\Big(\eta\xi\pm\frac{i \eta \lambda}{2\pi}\Big)\\
&\ll_{\eta_0,N}\min\Big(t,\frac1{|\xi|},\frac1{|\xi|(\eta|\xi|)^N}\Big).
\end{align*}

The function $g^{\lambda}_{t,\eta}:\mathbb{R}\to\mathbb{R}$ defined in \eqref{def-modified-g}
is smooth, even, and compactly supported, so it may be used in the trace formula of Theorem~\ref{EvenTrace}.
It also satisfies $\hat{g}^{\lambda}_{t,\eta}(0)=2(t+\eta)\hat{\psi}(i\eta \lambda/2\pi)$, $g^{\lambda}_{t,\eta}(0)=2\hat{\psi}(i\lambda\eta/2\pi)= \mathrm{O}_{\eta_0}(1)$,
$(g^{\lambda}_{t,\eta})''(0)=0$. Running the proof of Lemma \ref{trace-estimates-lemma}
with $g^{\lambda}_{t,\eta}$ in place of $g_{y,\eta}$ gives
\begin{equation}
\label{twisted-asymp}
T_n^{\cos}[g^{\lambda}_{t, \eta}]=\delta_0(n)\int_{-\infty}^{\infty}g^{\lambda}_{t,\eta}(u)\,\dd\varpi_{\Gamma}^{\ast}(u)
+\mathrm{O}_{\Gamma,\eta_0}\left(\frac1{\eta^2}+(1+n^2)\Big(\log^{\ast}\frac1{\eta}+t\Big)\right).
\end{equation}
Defining analogously $h_{t,\eta}^{\lambda}=\chi_{[-\eta,t]}\star\psi_{\eta}^{\lambda}\,-\,\chi_{[-t,\eta]}\star\psi_{\eta}^{-\lambda}$, we find as in Lemma~\ref{trace-estimates-lemma} that
\begin{equation}
\label{twisted-asymp-2}
T_n^{\sin}[h^{\lambda}_{t, \eta}]=\mathrm{O}_{\Gamma,\eta_0}\left(\frac1{\eta^2}+(1+n^2)\Big(\log^{\ast}\frac1{\eta}+t\Big)\right)
\end{equation}
and note that $h_{t,\eta}^{\lambda}=g_{t,\eta}^{\lambda}\cdot\sgn$ outside $[-2\eta_0,2\eta_0]$.

Note that $\|g^{\lambda}_{t,\eta}\|_{\infty}\leqslant 2\hat{\psi}(i\eta\lambda/2\pi)=\mathrm{O}_{\eta_0}(1)$ and $\mathrm{supp}\,g^{\lambda}_{t,\eta}\subseteq [-t-\eta,t+\eta]$, and recall the sum $S_n^P[g^{\lambda}_{t,\eta}]$ defined in \eqref{SnTn-smooth}.
Combining \eqref{twisted-asymp}, \eqref{twisted-asymp-2}, and Lemma~\ref{weights-primitivity-smooth}, we may write
\begin{gather}
S_n^P[g^{\lambda}_{t,\eta}]=\sump_{[\gamma]}\ell(\gamma)g^{\lambda}_{t,\eta}(\ell(\gamma))e^{-\ell(\gamma)+in\hol(\gamma)}
=\delta_0(n)\int_{-\infty}^{\infty}g^{\lambda}_{t,\eta}(u)\,\dd \varpi^{\ast}_{\Gamma}(u)+s_n^P[g^{\lambda}_{t,\eta}],\nonumber\\
s_n^P[g^{\lambda}_{t,\eta}]=\mathrm{O}_{\Gamma,\eta_0}\left(\frac1{\eta^2}+(1+n^2)\Big(\log^{\ast}\frac1{\eta}+t\Big)\right). \label{snPg-estimate}
\end{gather}
From the definition \eqref{def-modified-g}, we have that
\begin{align*}
(\mathrm{d}g^{\lambda}_{t,\eta}/\mathrm{d}t)(\ell)
&=\psi^{\lambda}_{\eta}(\ell-t)+\psi_{\eta}^{-\lambda}(\ell+t)\\
&=\psi_{\eta}(\ell-t)e^{\lambda(\ell-t)}+\psi_{\eta}(\ell+t)e^{-\lambda(\ell+t)}.
\end{align*}
Therefore,
\begin{equation}
\label{Kn-tilde-nontilde}
\widetilde{K}_n^{\lambda}[g_{y,\eta}]:=
\int_{\eta_0+}^y e^{\lambda t} \frac{\mathrm{d}}{\mathrm{d}t} S_n^P[g_{t,\eta}^\lambda] \, \dd t=K_n^\lambda[g_{y, \eta}]+\mathrm{O}_{\Gamma,\eta_0}(1),
\end{equation}
where
\[ K_n^{\lambda}[g_{y,\eta}] = \sump_{[\gamma]}\ell(\gamma) g_{y, \eta}(\ell(\gamma)) e^{(\lambda - 1)\ell(\gamma) + i n \hol \gamma} \]
and the $\mathrm{O}_{\Gamma,\eta_0}(1)$ term accounts for classes $[\gamma]$ with $\ell(\gamma)\leqslant 2\eta_0$ (in particular, this harmless term may be omitted if $\eta_0\leqslant\frac12\eta_0(\Gamma)$).
Following the argument in \eqref{Kny-rewritten}, we separate the integral representing $\widetilde{K}_n^{\lambda}[g_{y,\eta}]$ as
\[ \widetilde{K}_n^{\lambda}[g_{y,\eta}]=\delta_0(n)\int_{\eta_0}^{y}e^{\lambda t}\frac{\mathrm{d}}{\mathrm{d}t}\int_{-\infty}^{\infty}g_{t,\eta}^{\lambda}(u)\,\dd\varpi_{\Gamma}^{\ast}(u)\,\dd t+\int_{\eta_0+}^{y}e^{\lambda t}\frac{\mathrm{d}}{\mathrm{d}t}s_n^P[g^{\lambda}_{t,\eta}]\,\dd t, \]
where the first term equals, with notation as in \eqref{Kny-part1},
\begin{equation}
\label{first-term-twisted}
\delta_0(n)\int_2^{\infty}g_{y,\eta}(u)e^{\lambda u}\,\dd\varpi_{\Gamma}^{\ast}(u)+\mathrm{O}_{\Gamma,\eta_0}\Big(y+\mathbf{1}_{\lambda>\nu_k+1/y}\frac{e^{(\lambda-\nu_k)y}}{\lambda-\nu_k}\Big).
\end{equation}

For the second term, we use integration by parts and estimate \eqref{snPg-estimate} to find that it equals
\begin{equation}
\label{second-term-twisted}
e^{\lambda t}s_n^P[g^{\lambda}_{t,\eta}]\bigg|_{\eta_0+}^{y}-\lambda\int_{\eta_0}^{y}s_n^P[g^{\lambda}_{t,\eta}]e^{\lambda t}\,\dd t
\ll_{\Gamma,\eta_0}e^{\lambda y} \bigg(\frac1{\eta^2}+(1+n^2)\Big(\log^{\ast}\frac1{\eta}+y\Big) \bigg).
\end{equation}

Combining \eqref{Kn-tilde-nontilde}, \eqref{first-term-twisted} and \eqref{second-term-twisted} gives an asymptotic for $K_n^{\lambda}[g_{y,\eta}]$. Finally, we recover the desired sum $K_n[g_{y,\eta}]$ defined in \eqref{holonomy-character-sums} as
 \begin{align}
 K_n[g_{y,\eta}]
&=\sump_{[\gamma]}e^{in\hol(\gamma)}g_{y,\eta}(\ell(\gamma))\ell(\gamma)e^{-\ell(\gamma)}\bigg[\int_0^1e^{\lambda\ell(\gamma)}\,\dd\lambda+\frac1{\ell(\gamma)}\bigg]\nonumber\\
&=\int_0^1 K_n^{\lambda}[g_{y,\eta}]\,\dd\lambda+\mathrm{O}_{\Gamma,\eta_0}\Big(\frac{e^y}y\Big)\nonumber\\
&=\delta_0(n)\int_2^{\infty}g_{y,\eta}(u)\,\dd\varpi_{\Gamma}(u)+
\mathrm{O}_{\Gamma, \eta_0}
\bigg(\frac{e^y}y \bigg(\frac1{\eta^2}+(1+n^2)\Big(\log^{\ast}\frac1{\eta}+y\Big) \bigg)\bigg),
 \label{Kn-obtained}
 \end{align}
keeping in mind the definition \eqref{ei-varpi}, and using, for example, the already proved asymptotic \eqref{PGT-for-record} for $K_0(y)$ at the second step. We may assume that $\log^{\ast}(1/\eta)=\mathrm{O}_{\Gamma}(y)$ since otherwise (in light of \eqref{ei-varpi} and \eqref{PGT-for-record}) the error term for $K_n[g_{y,\eta}]$ in \eqref{hol-char-sum-result} clearly dominates all other terms. With this, \eqref{Kn-obtained} completes the proof of Proposition~\ref{holonomy-sums-prop}.
\end{proof}

\begin{remark}
\label{remark-holonomy-sums}
As already remarked, Proposition~\ref{holonomy-sums-prop} contains as a special case the Prime Geodesic Theorem in the form
\begin{equation}
\label{PGT-for-record}
\pi_{\Gamma}(y)=K_0(y)=\int_2^y\dd\varpi_{\Gamma}(u)+\mathrm{O}_{\Gamma}\Big(\frac{e^{5y/3}}y\Big),
\end{equation}
which we record here for future reference.

It is also instructive to consider how the asymptotic obtained in Proposition~\ref{holonomy-sums-prop} for $K_n[g_{y,\eta}]$ evolves as $\eta>0$ varies from $\eta$ of constant size, a case we may think of as a model of summation with a smooth cutoff in length on a $\mathrm{O}(1)$ scale, down to $\eta=e^{-y/3}$, which essentially corresponds to the sharp cutoff in $K_n(y)$ (cf.~proof of Proposition ~\ref{trace-estimates-sharp}). For $\eta=\eta_0$ of constant size, Proposition~\ref{holonomy-sums-prop} states that
\begin{equation}
\label{Kn-eta0-fixed}
K_n[g_{y,\eta_0}]=\delta_0(n)\int_2^{\infty}g_{y,\eta_0}(u)\,\dd\varpi_{\Gamma}(u)+\mathrm{O}_{\Gamma,\eta_0}\big((n^2+1)e^y\big).
\end{equation}
Thus, for a fixed $n\neq 0$, the sum $K_n[g_{y,\eta_0}]$ consisting of $\sim_{\Gamma, \eta_0}e^{2y}/2y$ terms of unit size exhibits essentially square-root cancellation as $y\to\infty$. In fact, power-saving cancellation in $K_n[g_{y,\eta}]$ and $K_n(y)$ persists in the range $n\ll e^{(1/2-\delta)y}$, a statement which should be compared to the range of uniformity in the prime number theorem for arithmetic progressions to large moduli.

As $\eta$ decreases and the cutoff in $K_n[g_{y,\eta}]$ becomes steeper, the first error term (which is independent of $n$) becomes more pronounced and, for $\eta=e^{-y/3}$, essentially recovers the first error term in the asymptotic for the sharp count $K_n(y)$. This term, which dominates for a fixed $n\in\mathbb{Z}$ and $y\to\infty$, is rooted in the passage to the sharp count. The second error term detects for large $n\in\mathbb{Z}$ the influence of the oscillating holonomy factor $e^{in\hol(\gamma)}$.
\end{remark}

\section{Ambient prime geodesic theorems}\label{ambient-pgt-section}
In this section, we prove our principal results providing counts for prime geodesics on $M$ with control on their length and holonomy simultaneously.

The first of the two main results of this section is the Ambient Prime Geodesic Theorem (Theorem ~\ref{sharpsharp}),
 which features a sharp cutoff both in length $\ell(\gamma)\leqslant y$
 and holonomy, as in \eqref{SW-effective}. For many analytic purposes, including the existence and properties of various limiting distributions (see, for example, Corollary~\ref{equid-corr} of Theorem~\ref{SarnakTheorem}), smooth-cutoff results suffice.  Since the passage to the sharp count is the leading contributor to error terms, and to emphasize the parallel between length and holonomy aspects, we provide $4=2\times 2$ propositions in \S\ref{smooth-count-subsec} and \S\ref{ambient-pgt-sharp-section},
featuring each combination of smooth/sharp counts in length/holonomy, with explicit error terms depending on the steepness of the smooth cutoff. Further, in the concluding \S\ref{APGT-shrinking-subsec}, we consider the short-range ambient counting problems and show how a consistent ambiental passage from smooth to sharp counting leads to stronger corresponding short-range  counts with smooth/sharp cutoffs, including our second main result, Theorem~\ref{short-range-ambient-theorem}.

To effectuate the transition from a smooth to sharp cutoff in holonomy, we will use, for an arbitrary interval $J\subseteq\mathbb{R}/2\pi\mathbb{Z}$ and $0<\eta'\leqslant 2\pi$ 
 the function $f_{J,\eta'}:\mathbb{R}/2\pi\mathbb{Z}\to\mathbb{R}$ given by
\begin{equation}
\label{f-I-eps}
f_{J,\eta'}(t)=\sum_{n\in\mathbb{Z}}\int_J\psi_{\eta'}(t+2n\pi-\theta)\,\dd\theta,
\end{equation}
with $\psi_{\eta'}$ as in \eqref{psi-eta-def}. In other words, $f_{J,\eta'}$ is the $2\pi\mathbb{Z}$-periodic convolution $\chi_J\star\psi_{\eta'}$ and plays for holonomy the role of $g_{y,\eta}$ and $h_{y,\eta}$ of \eqref{smooth-cutoff-functions} for lengths. In particular, $f_{[\theta,\theta'],\eta'}$ is a smooth, non-negative cutoff function of height 1, which is supported on $[\theta-\eta',\theta'+\eta']+2\pi\mathbb{Z}$ and agrees with $\chi_{[\theta,\theta']}$ outside the set $([\theta-\eta',\theta+\eta']\cup[\theta'-\eta',\theta'+\eta'])+2\pi\mathbb{Z}$.

\subsection{Smooth count}
\label{smooth-count-subsec}
We prove our prime geodesic theorems in holonomy classes by spectrally decomposing the holonomy and then invoking estimates on holonomy character sums in Proposition~\ref{holonomy-sums-prop}. We observed in Proposition \ref{holonomy-sums-prop} and Remark~\ref{remark-holonomy-sums} that, when lengths are weighted with a smooth function $g_{y,\eta}$ with not too small $\eta>0$ (say, of fixed size), so that the cutoff is not too steep, one obtains strong asymptotics for holonomy character sums with very modest error terms, such as essentially square-root cancellation in \eqref{Kn-eta0-fixed}.

Our first proposition is the baseline count for sampling geodesics with a smooth function in both the length and the holonomy. For cutoffs of fixed steepness (equivalently with fixed ``uncertainty windows,'' which we denote by $\eta,\eta'>0$) in both the length and the holonomy, Proposition~\ref{SmoothSmooth} features an error term of essentially square-root strength, well sharper than the sharp-cutoff counts in either direction in Proposition~\ref{SharpHolSmoothLength} and Theorems~\ref{SarnakTheorem} and \ref{sharpsharp} below.

\begin{prop}
\label{SmoothSmooth}
Let $\Gamma<\PSL_2\mathbb{C}$ be a discrete, co-compact, torsion-free subgroup. Then, for every smooth function $f:\mathbb{R}/2\pi\mathbb{Z}\to\mathbb{C}$ and every $y>0$ and $0<\eta\leqslant\eta_0$,
\begin{align*}
\pi_{\Gamma}(g_{y,\eta},f)&:=
\sump_{[\gamma]} f(\hol(\gamma)) g_{y,\eta}(\ell(\gamma))\\
&= \frac1{2\pi}\int_0^{2\pi}f(\theta)\,\dd\theta\cdot\int_2^{\infty}g_{y,\eta}(u)\,\dd\varpi_{\Gamma}(u)+ \mathrm{O}_{\Gamma,\eta_0}\Big(\frac{e^y}{y\eta^2}\|\hat{f}\|_1+e^y\|\hat{f}\|_{2,1}\Big)\\
&=\frac1{2\pi}\int_0^{2\pi}f(\theta)\,\dd\theta\cdot\pi_{\Gamma}(g_{y,\eta})+ \mathrm{O}_{\Gamma,f,\eta}(e^y),
\end{align*}
where $\varpi_{\Gamma}$ is as in \eqref{ei-varpi} and $\|\hat{f}\|_{2,1}=\|\hat{f}\|_1+\|\widehat{f''}\|_1$.
In particular, for every interval $J\subseteq\mathbb{R}/2\pi\mathbb{Z}$, $0<\eta'\leqslant 2\pi$, and $f_{J,\eta'}:\mathbb{R}/2\pi\mathbb{Z}\to\mathbb{R}$ as in \eqref{f-I-eps},
\[ \pi_{\Gamma}(g_{y,\eta},f_{J,\eta'})
=\frac{|J|}{2\pi}\int_2^{\infty}g_{y,\eta}(u)\dd\varpi_{\Gamma}(u)+\mathrm{O}_{\Gamma,\eta_0}\Big(\frac{e^y}{y\eta^2}\log^{\ast}\frac1{\eta'}+\frac{e^y}{\eta'{}^2}\Big). \]
\end{prop}

\begin{proof}
	Using the Fourier expansion $f(\theta) = (1/2\pi)\sum_{n\in\mathbb{Z}}\hat f(n) e^{in\theta}$ and applying Proposition~\ref{holonomy-sums-prop} to estimate the resulting holonomy character sums $K_n[g_{y,\eta}]$ defined in \eqref{holonomy-character-sums}, we find that
\begin{align*}
\sump_{[\gamma]} f(\hol(\gamma)) g_{y,\eta}(\ell(\gamma))&=\frac1{2\pi}\sump_{[\gamma]} g_{y,\eta}(\ell(\gamma)) \sum_{n\in\mathbb{Z}}\hat{f}(n)e^{in\hol\gamma}=\frac1{2\pi}\sum_{n\in\mathbb{Z}}\hat{f}(n)K_n[g_{y,\eta}]\\
&=\frac1{2\pi}\hat{f}(0)\pi_{\Gamma}(g_{y,\eta})+\mathrm{O}_{\Gamma,\eta_0}\bigg(\sum_{n\neq 0}|\hat{f}(n)|e^y\Big(\frac1{y\eta^2}+n^2+1\Big)\bigg),
\end{align*}
which completes the proof for a general smooth $f:\mathbb{R}/2\pi\mathbb{Z}\to\mathbb{C}$, keeping in mind that $\pi_{\Gamma}(g_{y,\eta})=K_0[g_{y,\eta}]$ and Proposition~\ref{holonomy-sums-prop}. For the specific function $f_{J,\eta'}$ defined in \eqref{f-I-eps}, we compute its Fourier coefficients by the usual unfolding argument as
\[ \hat{f}_{J,\eta'}(n)=\hat{\chi}_J(n)\hat{\psi}(\eta' n)=\frac{e^{-in\theta'}-e^{-in\theta}}{-in}\hat{\psi}(\eta' n), \]
where $J=[\theta,\theta']$. 	Using the Schwartz estimate $|\hat{\psi}(\eta' n)|\ll_k1/(1+\eta'|n|)^k$, we can bound
\begin{equation}
\label{fourier-sum-bounds}
\begin{alignedat}{7}
	&\sum_{n\in\mathbb{Z}}|\hat{f}_{J,\eta'}(n)|&&\ll|J|+\sum_{n\neq 0}\frac{|\hat{\psi}(\eta' n)|}n&&\ll\sum_{n\leqslant 1/\eta'}\frac1n+\sum_{n>1/\eta'}\frac1{\eta'{}^2n^3}&&\ll\log^{\ast}\frac1{\eta'},\\
	&\sum_{n\in\mathbb{Z}}(1+n^2)|\hat{f}_{J,\eta'}(n)|&&\ll|J|+\sum_{n\neq 0}n|\hat{\psi}(\eta' n)|&&\ll\sum_{n\leqslant 1/\eta'}n+\sum_{n>1/\eta'}\frac{n}{\eta'{}^4n^4}&&\ll\frac1{\eta'{}^2}.
\end{alignedat}
\end{equation}
This completes the proof.
\end{proof}

\subsection{Passage to sharp counts and Prime Geodesic Theorems}\label{ambient-pgt-sharp-section}
In this section, we replace the smooth cutoff in Proposition~\ref{SmoothSmooth} by a sharp cutoff in one or both of the length and holonomy. Theorem~\ref{SarnakTheorem} features the familiar sharp cutoff in length, a hallmark of a traditional Prime Geodesic Theorem, and implies effective equidistribution of holonomy in short intervals of length. To stress the conceptual symmetry between the two parameters, we also prove Proposition~\ref{SharpHolSmoothLength}, an asymptotic count with a sharp cutoff in holonomy and smoothly sampled length. Finally, in Theorem~\ref{sharpsharp} and its Corollary~\ref{ambient-short-range}, we prove ambient prime geodesic counts with a sharp cutoff in both length and holonomy.

\begin{theorem}
	\label{SarnakTheorem}
	Let $\Gamma<\PSL_2\mathbb{C}$ be a discrete, co-compact, torsion-free subgroup, and let $f:\mathbb{R}/2\pi\mathbb{Z}\to\mathbb{C}$ be an arbitrary smooth function. Then, for $y>0$,
	\begin{align*}
	\pi_{\Gamma}(y,f):=\sump_{\ell(\gamma) \leqslant y} f(\hol(\gamma))
	&=\frac1{2\pi}\int_0^{2\pi}f(\theta)\,\dd\theta\cdot\int_2^y\dd\varpi_{\Gamma}(u)+\mathrm{O}_{\Gamma}\Big(\|\hat{f}\|_1\frac{e^{5y/3}}y+\|\widehat{f''}\|_1e^y\Big)\\
	&=\frac1{2\pi}\int_0^{2\pi}f(\theta)\,\dd\theta\cdot\pi_{\Gamma}(y)+\mathrm{O}_{\Gamma,f}\Big(\frac{e^{5y/3}}y\Big),
	\end{align*}
where $\varpi_{\Gamma}$ is as in \eqref{ei-varpi}.
		In particular, for every interval $J\subseteq\mathbb{R}/2\pi\mathbb{Z}$, $0<\eta'\leqslant 2\pi$, and $f_{J,\eta'}:\mathbb{R}/2\pi\mathbb{Z}\to\mathbb{R}$ as in \eqref{f-I-eps},
	\[ \pi_{\Gamma}(y,f_{J,\eta'})
	=\frac{|J|}{2\pi}\pi_{\Gamma}(y)+\mathrm{O}_{\Gamma}\Big(\frac{e^{5y/3}}y\log^{\ast}\frac1{\eta'}+\frac{e^y}{\eta'{}^2}\Big). \]
\end{theorem}

\begin{proof}
	Using the Fourier expansion $f(\theta) = (1/2\pi)\sum_{n\in\mathbb{Z}}\hat f(n) e^{in\theta}$ and applying Proposition~\ref{holonomy-sums-prop} to estimate the resulting holonomy character sums $K_n(y)$ defined in \eqref{holonomy-character-sums}, we find that
	\begin{align*}
	\sump_{\ell(\gamma) \leqslant y} f(\hol(\gamma))
	& =  \frac1{2\pi}\sump_{\ell(\gamma) \leqslant y} \sum_{n\in\mathbb{Z}}\hat f(n) e^{in \hol \gamma}
	= \frac1{2\pi}\sum_{n\in\mathbb{Z}} \hat f(n) K_n(y)\\
	& = \frac1{2\pi}\hat f(0)\pi_{\Gamma}(y) + \mathrm{O}_{\Gamma}\Big(\frac{e^{5y/3}}y\sum_{n \neq 0} |\hat f(n)| + e^y \sum_{n \neq 0} n^2 |\hat f(n)|\Big),
	\end{align*}
	which completes the proof for a general smooth $f:\mathbb{R}/2\pi\mathbb{Z}\to\mathbb{C}$, keeping in mind \eqref{PGT-for-record}. The final claim follows by specializing these bounds to $f_{J,\eta'}$ and using estimates \eqref{fourier-sum-bounds}.
	\end{proof}

For a fixed smooth $f:\mathbb{R}/2\pi\mathbb{Z}\to\mathbb{C}$, Theorem~\ref{SarnakTheorem} recovers \cite[Theorem 1]{SarWa} in the present setting of compact hyperbolic 3-manifolds. An immediate corollary of Theorem~\ref{SarnakTheorem}, coupled with the Prime Geodesic Theorem~\eqref{PGT-for-record}, is the following equidistribution statement.

\begin{corr}[Equidistribution of holonomy]
\label{equid-corr}
Let $\Gamma<\PSL_2\mathbb{C}$ be a discrete, co-compact, torsion-free subgroup. Then, the holonomy of geodesics on $\Gamma\backslash\mathbb{H}^3$ of length $\ell(\gamma)\leqslant y$ is equidistributed in $\mathbb{R}/2\pi\mathbb{Z}$ as $y\to\infty$. In fact, given any collection of intervals $I_n=[y_n',y_n]$ ($0\leqslant y_n'\leqslant y_n$) satisfying $|I_n|\,/\,e^{-y_n/3}\to\infty$,
\[ \frac1{\pi_{\Gamma}(I_n)}\sum_{\ell(\gamma)\in I_n}\delta_{\{\hol(\gamma)\}}
\,\xrightarrow{\,\mathrm{weak}^{\ast}\,}\,\frac1{2\pi}\dd\theta_{\mathbb{R}/2\pi\mathbb{Z}}\quad (n\to\infty). \]
\end{corr}

Now we transition from a smooth function on the holonomy to a sharp cutoff, while still using a smooth function on the length. We do so using the smooth cutoff function $f_{J,\eta'}$ defined in \eqref{f-I-eps} and then optimizing the choice of $\eta'$.

\begin{prop}
\label{SharpHolSmoothLength}
Let $\Gamma<\PSL_2 \mathbb{C}$ be a discrete, co-compact, torsion-free subgroup. Then, for every interval $J= [\theta, \theta'] \subseteq \mathbb{R} / 2 \pi \mathbb{Z}$ and every $y>0$ and $0<\eta\leqslant\eta_0$,
\begin{align*}
\pi_{\Gamma}(g_{y,\eta},J):=\sump_{\hol(\gamma) \in J} g_{y,\eta}(\ell(\gamma))
&=\frac{|J|}{2\pi}\int_2^{\infty}g_{y,\eta}(u)\,\dd\varpi_{\Gamma}(u)+\mathrm{O}_{\Gamma,\eta_0}\Big(\frac{e^{5y/3}}{y^{2/3}}+\frac{e^y}{\eta^2}\Big)\\
&= \frac{|J|}{2\pi} \pi_{\Gamma}(g_{y,\eta}) +\mathrm{O}_{\Gamma,\eta_0}\Big(\frac{e^{5y/3}}{y^{2/3}}+\frac{e^y}{\eta^2}\Big).
\end{align*}
\end{prop}

\begin{proof}
Our starting point is Proposition~\ref{SmoothSmooth}, with a parameter $0<\eta'\leqslant 2\pi$ to be suitably chosen soon, which yields the estimate
\[ \pi_{\Gamma}(g_{y,\eta},f_{J,\eta'})
=\frac{|J|}{2\pi}\int_2^{\infty}g_{y,\eta}(u)\dd\varpi_{\Gamma}(u)+\mathrm{O}_{\Gamma,\eta_0}\Big(\frac{e^y}{y\eta^2}\log^{\ast}\frac1{\eta'}+\frac{e^y}{\eta'{}^2}\Big). \]
According to definition \eqref{f-I-eps}, the smooth cutoff function $f_{J,\eta'}$ for $J=[\theta,\theta']$ agrees with $\chi_J$ outside the set $([\theta-\eta',\theta+\eta']\cup[\theta'-\eta',\theta'+\eta'])+2\pi\mathbb{Z}$, on which $|f_{J,\eta'}-\chi_J|=\mathrm{O}(1)$. We also recall that, according to definition \eqref{smooth-cutoff-functions}, the cutoff function $g_{y,\eta}$ is supported on $[-y-\eta,y+\eta]$ and satisfies $|g_{y,\eta}|\leqslant 1$. Therefore,
\begin{align*}
\big|\pi_{\Gamma}(g_{y,\eta},J)-\pi_{\Gamma}(g_{y,\eta},f_{J,\eta'})\big|
&=\bigg|\sump_{[\gamma]}(\chi_J-f_{J,\eta'})(\hol(\gamma))g_{y,\eta}(\ell(\gamma))\bigg|\\
&\ll\sump_{\substack{\ell(\gamma)\leqslant y+\eta\\ \theta-\eta'\leqslant\hol(\gamma)\leqslant\theta+\eta'}}1+\sump_{\substack{\ell(\gamma)\leqslant y+\eta\\ \theta'-\eta'\leqslant\hol(\gamma)\leqslant\theta'+\eta'}}1.
\end{align*}
The latter terms are of the form ready to be estimated using Lemma~\ref{passage-to-sharp-holonomy}, which for every $\theta\in\mathbb{R}/2\pi\mathbb{Z}$ yields
\begin{equation}
\label{from-sharp-hol-1}
\sump_{\substack{\ell(\gamma)\leqslant y+\eta\\ \theta-\eta'\leqslant\hol(\gamma)\leqslant\theta+\eta'}}1
\ll
\frac{e^{y+\eta}}y\sum_{\substack{\ell(\gamma)\leqslant y+\eta\\ \theta-\eta'\leqslant\hol(\gamma)\leqslant\theta+\eta'}}\ell(\gamma_0)w(\gamma)\ll_{\Gamma,\eta_0}\eta'\frac{e^{2y}}y+\frac{e^y}{\eta'{}^2}.
\end{equation}

Using Proposition~\ref{SmoothSmooth} and the input from Lemma~\ref{passage-to-sharp-holonomy} in the form \eqref{from-sharp-hol-1}, we have that
\begin{align}
\label{difference}
&\bigg|\pi_{\Gamma}(g_{y,\eta},J)-\frac{|J|}{2\pi}\int_2^yg_{y,\eta}(u)\,\dd\varpi_{\Gamma}(u)\bigg|\\
\nonumber
&\qquad\leqslant \big|\pi_{\Gamma}(g_{y,\eta},J)-\pi_{\Gamma}(g_{y,\eta},f_{J,\eta'})\big|
+\bigg|\pi_{\Gamma}(g_{y,\eta},f_{J,\eta'})-\frac{|J|}{2\pi}\int_2^{\infty}g_{y,\eta}(u)\dd\varpi_{\Gamma}(u)\bigg|\\
\nonumber
&\qquad\ll_{\Gamma,\eta_0}\eta'\frac{e^{2y}}y+\frac{e^y}{\eta'{}^2}+\frac{e^y}{y\eta^2}\log^{\ast}\frac1{\eta'}.
\end{align}
Taking the admissible choice $\eta'=\min(y^{1/3}e^{-y/3},2\pi)$ completes the proof of the first claim, since for $y=\mathrm{O}(1)$ the left-hand side is trivially $\mathrm{O}_{\Gamma,\eta_0}(1)$.

The second claim follows immediately, keeping in mind that by Proposition~\ref{holonomy-sums-prop}
\[  \pi_{\Gamma}(g_{y,\eta})=K_0[g_{y,\eta}]=\int_2^{\infty}g_{y,\eta}(u)\,\dd\varpi_{\Gamma}(u)+\mathrm{O}_{\Gamma,\eta_0}\Big(\frac{e^y}{y\eta^2}+e^y\Big), \]
or alternatively by adapting the use of Proposition~\ref{SmoothSmooth} in \eqref{difference}.
\end{proof}

Our main result, the following Theorem~\ref{sharpsharp}, features a sharp cutoff in both length and holonomy. Such a theorem can be proved by passing from the remaining smooth to sharp cutoff in either Theorem~\ref{SarnakTheorem} or Proposition~\ref{SharpHolSmoothLength}; we choose the former route.

\begin{theorem}[Ambient prime geodesic theorem]
\label{sharpsharp}
Let $\Gamma<\PSL_2\mathbb{C}$ be a discrete, co-compact, torsion-free subgroup. Then, for every $y>0$ and every interval $J\subseteq\mathbb{R}/2\pi\mathbb{Z}$,

\begin{align*}
\pi_{\Gamma}(y,J)
&:=\big|\big\{[\gamma]^P:\ell(\gamma)\leqslant y,\,\hol(\gamma)\in J\big\}\big|\\
&=\iint_{[2,y]\times J}\dd\varpi_{\Gamma}(u)\frac{\dd\theta}{2\pi}+\mathrm{O}_{\Gamma}\big(e^{5y/3}\big)
=\frac{|J|}{2\pi}\pi_{\Gamma}(y)+\mathrm{O}_{\Gamma}\big(e^{5y/3}\big),
\end{align*}
where $\varpi_{\Gamma}$ is as in \eqref{ei-varpi}.
\end{theorem}

\begin{proof}
We begin with Theorem~\ref{SarnakTheorem}, with a parameter $0<\eta'\leqslant 2\pi$ to be suitably chosen soon. Recall from \eqref{f-I-eps} that the smooth cutoff function $f_{J,\eta'}$ for $J=[\theta,\theta']$ agrees with $\chi_J$ outside the set $([\theta-\eta',\theta+\eta']\cup[\theta'-\eta',\theta'+\eta'])+2\pi\mathbb{Z}$, on which $|f_{J,\eta'}-\chi_J|=\mathrm{O}(1)$. Arguing as in the proof of Proposition~\ref{SharpHolSmoothLength}, we find using Lemma~\ref{passage-to-sharp-holonomy} that
\begin{align}
\nonumber \big|\pi_{\Gamma}(y,J)-\pi_{\Gamma}(y,f_{J,\eta'})\big|
&\ll\sump_{\substack{\ell(\gamma)\leqslant y\\ \theta-\eta'\leqslant\hol(\gamma)\leqslant\theta+\eta'}}1+\sump_{\substack{\ell(\gamma)\leqslant y\\ \theta'-\eta'\leqslant\hol(\gamma)\leqslant\theta'+\eta'}}1\\
\label{from-sharp-hol}
&\ll\frac{e^y}y\sum_{\substack{\ell(\gamma)\leqslant y\\ \theta-\eta'\leqslant\hol(\gamma)\leqslant\theta+\eta'\\\text{or }\theta'-\eta'\leqslant\hol(\gamma)\leqslant\theta'+\eta'}}\ell(\gamma_0)w(\gamma)\ll_{\Gamma}\eta'\frac{e^{2y}}y+\frac{e^y}{\eta'{}^2}.
\end{align}

Using Theorem~\ref{SarnakTheorem} and the input from Lemma~\ref{passage-to-sharp-holonomy} in the form \eqref{from-sharp-hol}, we have that
\begin{align*}
&\bigg|\pi_{\Gamma}(y,J)-\iint_{[2,y]\times J}\dd\varpi_{\Gamma}(u)\frac{\dd\theta}{2\pi}\bigg|\\
&\qquad\leqslant \big|\pi_{\Gamma}(y,J)-\pi_{\Gamma}(y,f_{J,\eta'})\big|
+\bigg|\pi_{\Gamma}(y,f_{J,\eta'})-\frac{|J|}{2\pi}\int_{[2,y]}\dd\varpi_{\Gamma}(u)\bigg|\\
&\qquad\ll_{\Gamma}\eta'\frac{e^{2y}}y+\frac{e^y}{\eta'{}^2}+\frac{e^{5y/3}}y\log^{\ast}\frac1{\eta'}.
\end{align*}
Taking the admissible choice $\eta'=\min(e^{-y/3},2\pi)$
 completes the proof, noting that for $y=\mathrm{O}(1)$ the left-hand side is $\mathrm{O}_{\Gamma}(1)$, and keeping in mind \eqref{PGT-for-record}.
\end{proof}

\begin{remark}
Observe that Proposition \ref{SharpHolSmoothLength}, which features a sharp cutoff in holonomy and a smooth cutoff in length, parallels Theorem \ref{SarnakTheorem}, which has a smooth cutoff in holonomy and a sharp cutoff in length. This strengthens the perspective that length and holonomy should be counted as a pair. Compared to the smooth cutoff in Proposition~\ref{SmoothSmooth}, the smooth to sharp transition in either direction yields a significant contribution to the error term; however, the further passage to a sharp cutoff in both directions in Theorem~\ref{sharpsharp} leads to minimal or no increase in the error term.

Conceptually (and tracing through the proofs confirms this rigorously), this is so because using smooth functions \eqref{smooth-cutoff-functions} and \eqref{f-I-eps} to approximate sharp cutoffs essentially mimics what would be the use of Theorem~\ref{LL} with an (ineligible) sharp-cutoff test function $\chi_R(t)$ for a target rectangle $R\subseteq\mathbb{R}\times(\mathbb{R}/2\pi\mathbb{Z})$ of length and holonomy. The effect of a smooth cutoff of wall length $\eta\ll 1$ in one direction is that the essential support (in the sense of Schwartz decay) of the Fourier transform extends up to $\asymp 1/\eta$ in the dual spectral direction. Since typically  $\widehat{\chi_R}(\chi_{\nu,p})\asymp_R ((1+|\nu|)(1+|p|))^{-1}$, the contribution of the principal series terms is guided roughly (up to logarithmic factors) by the supremum of the Plancherel measure \eqref{plancherel} over the said spectral support, which is in turn symmetric in spectral parameters $\nu$ and $p$ and does not increase if the support extends in both parameters rather than just one.
\end{remark}

Already as an immediate consequence of Theorem~\ref{sharpsharp}, we obtain the following ambient short-range count for primitive geodesics on $M$, with the length and holonomy simultaneously restricted to short intervals. We will further improve upon this asymptotic in Theorem~\ref{short-range-ambient-theorem}.

\begin{corr}
\label{ambient-short-range}
Let $\Gamma<\PSL_2\mathbb{C}$ be a discrete, co-compact, torsion-free subgroup. Then, for any intervals $I=[y',y]$ ($0\leqslant y'\leqslant y$) and $J\subseteq\mathbb{R}/2\pi\mathbb{Z}$,
\begin{align*}
\pi_{\Gamma}(I,J)
&:=\big|\big\{[\gamma]^P:(\ell(\gamma),\hol(\gamma))\in  I \times J \big\}\big|\\
&= \iint_{I\times J}\dd\varpi_{\Gamma}(u)\frac{\dd\theta}{2\pi}+\mathrm{O}_{\Gamma}\big(e^{5y/3}\big).
\end{align*}
\end{corr}

\begin{remark}\label{shrinking-intervals-remark}
Corollary~\ref{ambient-short-range} provides an exponent-saving asymptotic for $\pi_{\Gamma}(I,J)$  as long as $|I\times J|\gg e^{(-1/3+\delta)y}$ for some $\delta>0$. We emphasize that the lengths $|I|$ and $|J|$ may be short \emph{independently of each other} in any regime satisfying this combined lower bound.

When Corollary~\ref{ambient-short-range} gives an asymptotic, this may be rewritten (with obvious shorthand notation) as $\pi_{\Gamma}(I,J)\sim_{\Gamma} \pi_{\Gamma}(I)\cdot |J|/(2\pi)\sim_{\Gamma} \Ei_{\Gamma}(I)\cdot |J|/(2\pi)$, which can in turn be restated as an effective equidistribution statement for either the lengths or the holonomies in shrinking intervals of lengths, holonomies, or both.

\end{remark}

\subsection{Ambient Prime Geodesic Theorems for shrinking intervals}
\label{APGT-shrinking-subsec}
Results of \S \ref{ambient-pgt-sharp-section}
establish structural parallels between the length and holonomy aspects in geodesic counting and indicate that the ``ambiental'' joint count of primitive classes $[\gamma]\in [\Gamma]$ according to the pair $(\ell(\gamma),\hol(\gamma))$ is perhaps the most natural counting object.
 In this section, we demonstrate how consistently executing ``ambiental'' passage from smooth to sharp count leads to ambient prime geodesic theorems which improve upon the results of \S\ref{ambient-pgt-sharp-section} when both intervals of length and holonomy are shrinking. To emphasize the symmetry and analogy to \S\ref{ambient-pgt-sharp-section}, in Proposition \ref{sharp-length-intervals-prop} we prove counts that are smooth (but possibly steep) in one of the parameters and sharp in the other, and in Theorem~\ref{short-range-ambient-theorem} our headline count that is sharp in both aspects.

For the smooth length cutoff, we use
\begin{equation}
\label{smooth-interval-cutoff-def}
g_{I,\eta}=\psi_{\eta}\star(\chi_{I} + \chi_{-I}),
\end{equation}
where $\psi_\eta$ is defined as in (\ref{psi-eta-def}) and $\chi_I$ and $\chi_{-I}$ are  characteristic functions, to sample the geodesics with length in an interval $I \subseteq [0,y]$. For the smooth holonomy cutoff on an interval $J \subseteq \mathbb{R}/2 \pi \mathbb{Z}$, we use the function $f_{J, \eta'}$ defined in (\ref{f-I-eps}), which is the periodization of $\psi_\eta \star \chi_J$. Further, for $I=[a,b]$ and $\eta>0$, we will write $I^{-}_{\eta}=[a-\eta,a+\eta]$, $I^{+}_{\eta}=[b-\eta,b+\eta]$, and $I_{\eta}=I^{-}_{\eta}\cup I\cup I^{+}_{\eta}=[a-\eta,b+\eta]$.

To begin, we modify a consequence of Proposition \ref{SmoothSmooth} to obtain a smooth count in both the length and holonomy aspects.

\begin{lemma}\label{smooth-smooth-intervals}
Let $\Gamma<\PSL_2\mathbb{C}$ be a discrete, co-compact, torsion-free subgroup. Then, for every two intervals $I =[y', y]$ ($0\leqslant y'< y$) and $J \subseteq \mathbb{R}/2\pi\mathbb{Z}$, and every $0< \eta\leqslant\eta_0$ and $0< \eta' \leqslant 2\pi$,
\[ \pi_{\Gamma}(g_{I,\eta},f_{J,\eta'})
=\frac{|J|}{2\pi}\int_2^{\infty}g_{I,\eta}(u)\,\dd\varpi_{\Gamma}(u)+\mathrm{O}_{\Gamma,\eta_0}\Big(\frac{e^y}{y\eta^2}\log^{\ast}\frac1{\eta'}+\frac{e^y}{\eta'{}^2}\Big), \]
where $g_{I,\eta}$ and $f_{J, \eta'}$ are smooth cutoff functions defined in \eqref{smooth-interval-cutoff-def} and \eqref{f-I-eps}, and $\varpi_\Gamma$ is the density \eqref{ei-varpi}.
\end{lemma}

\begin{proof} The result follows immediately from subtracting two instances of Proposition \ref{SmoothSmooth}, noting that $g_{I, \eta} = g_{y, \eta} - g_{y', \eta}$.
\end{proof}

The following corollary, which may be thought of as the ambient analogue of Lemmata~\ref{geodesics-upper-bound} and \ref{passage-to-sharp-holonomy}, will be used to give an upper bound on the boundary terms when transitioning from a smooth to a sharp cutoff. In our typical application of Corollary~\ref{short-range-upper-bound}, at least one of the intervals $I$ and $J$ will be very short, and we pick $\eta\asymp\min(|I|,\eta_0)$ or $\eta'\asymp |J|$.

\begin{corr}
\label{short-range-upper-bound}
Let $\Gamma<\PSL_2\mathbb{C}$ be a discrete, co-compact, torsion-free subgroup. Then, for every two intervals $I =[y', y]$ ($0\leqslant y'< y$) and $J \subseteq \mathbb{R}/2\pi\mathbb{Z}$, we have
\[\pi_\Gamma(I, J) \ll_{\Gamma, \eta_0} (|I| + \eta)(|J|+\eta') \frac{e^{2y}}{y} +\frac{e^y}{y\eta^2}\log^{\ast}\frac1{\eta'}+\frac{e^y}{\eta'{}^2} \]
for every choice of $0< \eta\leqslant\eta_0$ and $0< \eta'\leqslant 2\pi$.
\end{corr}

\begin{proof}
It is clear from definitions \eqref{f-I-eps} and \eqref{smooth-interval-cutoff-def} that
$\chi_{I\cup(-I)}
\leqslant g_{I_{\eta},\eta}$ and $\chi_{J+2\pi\mathbb{Z}}\leqslant f_{J_{\eta'},\eta'}$.
Combining this observation with Lemma~\ref{smooth-smooth-intervals}, and keeping in mind that $g_{I_{\eta},\eta}$ is supported inside $I_{2\eta}\cup (-I_{2\eta})$, we obtain
 \begin{align*}
	\pi_\Gamma(I, J) &
\leqslant\pi_{\Gamma}(g_{I_{\eta},\eta},f_{J_{\eta'},\eta'})\\
	&\ll_{\Gamma,\eta_0} \frac{|J_{\eta'}|}{2\pi}\int_2^{\infty}g_{I_\eta
		,\eta}(u)\,\dd\varpi_{\Gamma}(u)+\frac{e^y}{y\eta^2}\log^{\ast}\frac1{\eta'}+\frac{e^y}{\eta'{}^2}\\
	&\ll_{\Gamma, \eta_0}
	(|I| + \eta)(|J|+\eta') \frac{e^{2y}}{y} +\frac{e^y}{y\eta^2}\log^{\ast}\frac1{\eta'}+\frac{e^y}{\eta'{}^2}.
\qedhere
\end{align*}
\end{proof}

In the following proposition, we pass from the smooth result of Lemma \ref{smooth-smooth-intervals} to asymptotic counts of geodesics with a sharp cutoff in one of the length and holonomy parameters, retaining a smooth cutoff in the other. We pass to the sharp count using Corollary~\ref{short-range-upper-bound}, which
involves finding an upper bound in the two rectangular regions of ambiguity with at least one short side length, and thus improves  on the error term compared to Theorem~\ref{SarnakTheorem} and Proposition~\ref{SharpHolSmoothLength} (see Remark~\ref{concluding-remark} below).

\begin{prop}
\label{sharp-length-intervals-prop}
Let $\Gamma<\PSL_2\mathbb{C}$ be a discrete, co-compact, torsion-free subgroup, and let $\varpi_\Gamma$ be the density \eqref{ei-varpi}. Then, for any intervals $I=[y',y]$ ($0\leqslant y'< y$) and $J \subseteq\mathbb{R}/2\pi\mathbb{Z}$:
\begin{itemize}
\item[(a)]
For every $0< \eta'\leqslant 2\pi$ and for $f_{J, \eta'}$ as in \eqref{f-I-eps},
\begin{align*}
\pi_\Gamma(I, f_{J, \eta'})
&:=\sump_{\ell(\gamma) \in I} f_{J, \eta'}(\hol(\gamma))\\
&= \frac{|J|}{2 \pi}
\int_I\,\dd \varpi_\Gamma(u)
+ \mathrm{O}_{\Gamma}\Big((|J| + \eta')^{2/3} \frac{e^{5y/3}}{y}\Big(\log^{\ast}\frac{1}{\eta'}\Big)^{1/3} + \frac{e^y}{\eta'{}^2} \Big).
\end{align*}
\item[(b)]
For every $0 < \eta\leqslant\eta_0$ and for $g_{I, \eta}$ as in \eqref{smooth-interval-cutoff-def},
\begin{align*}
\pi_\Gamma(g_{I, \eta}, J) &:= \sump_{\hol(\gamma) \in J} g_{I,\eta}(\ell(\gamma))\\
&= \frac{|J|}{2 \pi}  \int_{2}^\infty g_{I, \eta}(u)\,\dd \varpi_\Gamma(u) + \mathrm{O}_{\Gamma, \eta_0}\Big( (|I| + \eta)^{2/3} \frac{e^{5y/3}}{y^{2/3}} +  \frac{e^y}{\eta^2} 
\Big).
\end{align*}
\end{itemize}
\end{prop}

\begin{proof}
We approximate $\pi_\Gamma(I, f_{J, \eta'})$ 
using the smooth cutoff
count $\pi_\Gamma(g_{I, \eta}, f_{J, \eta'})$, with $0<\eta\leqslant\eta_0$ to be suitably chosen later.
These counts differ at most by the number of primitive classes $[\gamma]$ with $(\ell(\gamma),\hol(\gamma))\in (I^{-}_{\eta}\cup I^{+}_{\eta})\times J_{\eta'}$.
By Corollary \ref{short-range-upper-bound}, we have the bound 
	\begin{equation}\label{sharp-length-interval-bound}
	\begin{aligned}
	|\pi_\Gamma(I, f_{J, \eta'}) -\pi_\Gamma(g_{I, \eta}, f_{J, \eta'})| &
	\leqslant\pi_{\Gamma}(I^{-}_{\eta},J_{\eta'})+\pi_{\Gamma}(I^{+}_{\eta},J_{\eta'})\\
	& \ll_{\Gamma,\eta_0} \eta (|J| + \eta') \frac{e^{2y}}{y} + \frac{e^y}{y \eta^2} \log^{\ast} \frac{1}{\eta'} + \frac{e^y}{\eta'{}^2}.
	\end{aligned}
	\end{equation}
	
Combining the estimate of Lemma \ref{smooth-smooth-intervals} for $\pi_\Gamma( g_{I, \eta}, f_{J,\eta'})$ with the estimate \eqref{sharp-length-interval-bound} on the region of ambiguity, we have that
\begin{align*}
\pi_\Gamma(I, f_{J, \eta'}) &= \frac{|J|}{2\pi} \int_2^\infty g_{I, \eta}(u)\,\dd \varpi_\Gamma(u) + \mathrm{O}_{\Gamma,\eta_0}\Big(\eta(|J| + \eta')\frac{e^{2y}}{y} + \frac{e^y}{y \eta^2} \log^{\ast} \frac{1}{\eta'} + \frac{e^y}{\eta'{}^2} \Big)\\
&=\frac{|J|}{2\pi} \int_I\,\dd \varpi_\Gamma(u) + \mathrm{O}_{\Gamma,\eta_0}\Big(\eta(|J| + \eta')\frac{e^{2y}}{y} + \frac{e^y}{y \eta^2} \log^{\ast} \frac{1}{\eta'} + \frac{e^y}{\eta'{}^2} \Big).
\end{align*}
	
Taking, say, $\eta_0 =1$, the essentially optimal choice $\eta = \min\big(e^{-y/3} (\log^{\ast}\frac{1}{\eta'})^{1/3}/(|J| + \eta')^{1/3}, \eta_0 \big)$ completes the proof of (a).

The proof of (b) is similar: approximating $\pi_\Gamma(g_{I, \eta}, J)$ with the smooth count
$\pi_{\Gamma}(g_{I,\eta},f_{J,\eta'})$, with $0<\eta'\leqslant 2\pi$ to be suitably chosen later,
we have that $|\pi_\Gamma(g_{I, \eta}, J) - \pi_\Gamma(g_{I, \eta}, f_{J, \eta'})|
\leqslant\pi_{\Gamma}(I_{\eta},J^{-}_{\eta'})+\pi_{\Gamma}(I_{\eta},J^{+}_{\eta'})$,
so that combining Corollary~\ref{short-range-upper-bound} and Lemma \ref{smooth-smooth-intervals} we conclude
\[\pi_\Gamma(g_{I, \eta}, J) =  \frac{|J|}{2\pi} \int_{2}^\infty g_{I, \eta}(u)\,\dd \varpi_\Gamma(u)  + \mathrm{O}_{\Gamma,\eta_0}\Big( (|I| + \eta) \eta' \frac{e^{2y}}{y} + \frac{e^y}{y \eta^2} \log^{\ast}\frac1{\eta'} + \frac{e^y}{\eta'{}^2}\Big). \]
Choosing $\eta' = \min\big(y^{1/3} e^{-y/3}/(|I| + \eta)^{1/3}, 2\pi\big)$ completes the proof.
\end{proof}

Finally, we present our main theorem, which provides a count of length and holonomy in intervals $I$ and $J$, respectively. When $I = [0,y]$ and $J$ is fixed, this recovers Theorem \ref{sharpsharp}. However, when the lengths of $I$ and $J$ are shrinking, we have a significant improvement.

\begin{theorem}[Ambient short-range prime geodesic theorem]
\label{short-range-ambient-theorem}
Let $\Gamma<\PSL_2\mathbb{C}$ be a discrete, co-compact, torsion-free subgroup. Then, for any intervals $I=[y',y]$ ($0\leq  y'< y$) and $J\subseteq\mathbb{R}/2\pi\mathbb{Z}$,
\begin{align*}
\pi_{\Gamma}(I,J)
&:=\big|\big\{[\gamma]^P:(\ell(\gamma),\hol(\gamma))\in I\times J \big\}\big|\\
&=\iint_{I\times J}\dd\varpi_{\Gamma}(u)\frac{\dd\theta}{2\pi}+\mathrm{O}_{\Gamma}\Big( (|I| + |J|)^{2/3} \frac{e^{5y/3}}{y^{2/3}} + \frac{e^{3y/2}}{y^{1/2}}\Big),
\end{align*}
where $\varpi_\Gamma$ is the density \eqref{ei-varpi}.
\end{theorem}

\begin{proof}
We use the smooth count $\pi_{\Gamma}(g_{I,\eta},f_{J,\eta'})$, with parameters $0<\eta\leqslant\eta_0$ and $0<\eta'\leqslant 2\pi$ to be suitably chosen later, to approximate the sharp cutoff count $\pi_\Gamma(I,J)$. 
Using Corollary~\ref{short-range-upper-bound} to estimate the boundary terms,
as in the proof of Proposition \ref{sharp-length-intervals-prop}, we obtain
\begin{align*}
|\pi_\Gamma(I, J) - \pi_\Gamma(g_{I, \eta}, f_{J, \eta'})|
&\leqslant\pi_{\Gamma}(I^{-}_{\eta},J_{\eta'})+\pi_{\Gamma}(I^{+}_{\eta},J_{\eta'})+
\pi_{\Gamma}(I_{\eta},J^{-}_{\eta'})+\pi_{\Gamma}(I_{\eta},J^{+}_{\eta'})\\
&\ll_{\Gamma} (\eta |J| + \eta' |I| + \eta \eta') \frac{e^{2y}}{y} +\frac{e^y}{y\eta^2}\log^{\ast}\frac1{\eta'}+\frac{e^y}{\eta'{}^2}.
\end{align*}
	
Then, using Lemma \ref{smooth-smooth-intervals} for $\pi_\Gamma(g_{I, \eta}, f_{J, \eta'})$, we achieve the estimate
\begin{align}
\pi_\Gamma(I, J) &=\frac{|J|}{2\pi}\int_2^{\infty}g_{I,\eta}(u)\,\dd\varpi_{\Gamma}(u)+ \mathrm{O}_{\Gamma,\eta_0}\Big( (\eta |J| + \eta' |I| + \eta \eta') \frac{e^{2y}}{y} +  \frac{e^y}{y\eta^2}\log^{\ast}\frac1{\eta'}+\frac{e^y}{\eta'{}^2} \Big)\nonumber\\
\label{eta-etaprime}
&= \frac{|J|}{2\pi} \int_I\,\dd\varpi_{\Gamma}(u) + \mathrm{O}_{\Gamma,\eta_0}\Big( (\eta |J| + \eta' |I| + \eta \eta') \frac{e^{2y}}{y} +  \frac{e^y}{y\eta^2}\log^{\ast}\frac1{\eta'}+\frac{e^y}{\eta'{}^2} \Big).
\end{align}
	
To obtain the result, we use the essentially optimal choices
\[\eta = \min\Big(\frac{y^{1/3} e^{-y/3}}{|J|^{1/3}}, y^{1/4}e^{-y/4},\eta_0\Big), \quad \eta' = \min\Big(\frac{y^{1/3} e^{-y/3}}{|I|^{1/3}}, y^{1/4}e^{-y/4},2\pi\Big),\]
taking a fixed $\eta_0=1$. Here we note that a brief comparison of the three latter summands in the error term in \eqref{eta-etaprime} shows that indeed the term $e^{3y/2}/y^{1/2}$ in Theorem~\ref{short-range-ambient-theorem} is the best possible following \eqref{eta-etaprime}.
\end{proof}

\begin{remark}
\label{concluding-remark}
As already remarked, Theorem~\ref{short-range-ambient-theorem} recovers the long-range Theorem~\ref{sharpsharp} and its Corollary~\ref{ambient-short-range}. It provides a substantial improvement as soon as both $I$ and $J$ are short, which is particularly strong if the total boundary length $\asymp |I|+|J|$ is favorably small. For example, if $|I|\asymp |J|$, Theorem~\ref{short-range-ambient-theorem} gives a power-saving asymptotic as long as $|I|\asymp |J|\gg e^{(-1/4+\delta)y}$, with the error term potentially as small as $\mathrm{O}_{\Gamma,\delta}(e^{(3/2+\delta)y})$. In general, Theorem~\ref{short-range-ambient-theorem} yields a power-saving asymptotic whenever $|I\times J|^3/(|I|+|J|)^2\asymp |I\times J|\cdot\min(|I|^2,|J|^2)\gg e^{(-1+\delta)y}$ for some $\delta>0$.

Proposition~\ref{sharp-length-intervals-prop} similarly improves upon Theorem~\ref{SarnakTheorem} and Proposition~\ref{SharpHolSmoothLength}. For example, when sampling geodesics with a sharp cutoff in length in an interval of size $|I|\asymp 1$ and a ``mild'' holonomy cutoff with $\eta'\asymp |J|$, Proposition~\ref{sharp-length-intervals-prop} improves upon Theorem~\ref{SarnakTheorem} for all $e^{(-1/3+\delta)y}\leqslant |J|\leqslant e^{-\delta y}$, with the error term as good as $\mathrm{O}_{\Gamma}(e^{3y/2})$ for $|I|\asymp 1$ and $\eta',|J|\asymp e^{-y/4}$.
\end{remark}

\bibliographystyle{amsplain}
\bibliography{biblio} 

\end{document}